\documentclass[twoside,11pt,reqno]{amsart}
\usepackage{amsmath}
\usepackage{amssymb}
\usepackage{amscd}
\usepackage{mathrsfs}
\usepackage{epic}

\input xy
\xyoption{all}

\makeatletter

\@addtoreset{equation}{section}

\newtheorem{Proposition}{Proposition}[section]
\newtheorem{Lemma}[Proposition]{Lemma}
\newtheorem{Theorem}[Proposition]{Theorem}
\newtheorem{Corollary}[Proposition]{Corollary}

\theoremstyle{definition}

\newtheorem{Remark}[Proposition]{Remark}
\newtheorem{Example}[Proposition]{Example}

\newtheorem{Conjecture}[Proposition]{Conjecture}

\newbox\squ  
\setbox\squ=\hbox{\vrule width.6pt
   \vbox{\hrule height.6pt width.4em\kern1ex
         \hrule height.6pt}%
   \vrule width.6pt}

\def\0{{\bar{0}}}
\def\1{{\bar{1}}}

\def\lag{{\mathfrak{l}}}

\def\ch{{\operatorname{ch}}}
\def\ad{{\operatorname{ad}\,}}
\def\coad{{\operatorname{ad}^*}}
\def\ne{{\operatorname{ne}}}
\def\F{\operatorname{F}}
\def\Pr{\operatorname{Pr}}
\def\pr{\operatorname{pr}}
\def\res{\operatorname{res}}
\def\ann{\operatorname{Ann}}
\def\VA{\mathcal{V\!A}}
\def\op{\operatorname{op}}
\def\gr{\operatorname{gr}}

\def\tr{\operatorname{tr}}
\def\id{\operatorname{id}}

\def\RRow{\operatorname{Row}}

\def\C{{\mathbb C}}
\def\Z{{\mathbb Z}}

\def\im{{\operatorname{im}}}

\def\Tab{{\operatorname{Tab}}}

\def\la{{\lambda}}
\def\La{{\Lambda}}

\def\si{{\sigma}}
\def\al{{\alpha}}
\def\be{{\beta}}
\def\ga{{\gamma}}
\def\eps{{\varepsilon}}
\def\lan{\langle}
\def\ran{\rangle}

\def\CC{\operatorname{col}}
\def\RR{\operatorname{row}}
\def\N{\mathrm{N}}
\def\NE{\mathrm{NE}}
\def\E{\mathrm{E}}

\def\SS{\mathrm{S}}

\def\NW{\mathrm{NW}}

\def\a{\mathfrak a}
\def\b{\mathfrak b}
\def\g{\mathfrak g}
\def\h{\mathfrak h}
\def\k{\mathfrak k}
\def\m{\mathfrak m}
\def\n{\mathfrak n}
\def\p{\mathfrak p}
\def\q{\mathfrak q}

\def\t{\mathfrak t}
\def\u{\mathfrak u}

\def\gl{\mathfrak{gl}}

\def\aa{\mathbf a}
\def\bb{\mathbf b}
\def\cc{\mathbf c}

\def\OO{\mathcal O}
\def\D{\mathcal D}

\newdimen\Hoogte    \Hoogte=12pt    
\newdimen\Breedte   \Breedte=12pt   
\newdimen\Dikte     \Dikte=0.5pt    

\newenvironment{Young}{\begingroup
       \def\vr{\vrule height0.8\Hoogte width\Dikte depth 0.2\Hoogte}
       \def\fbox##1{\vbox{\offinterlineskip
                    \hrule height\Dikte
                    \hbox to \Breedte{\vr\hfill##1\hfill\vr}
                    \hrule height\Dikte}}
       \vbox\bgroup \offinterlineskip \tabskip=-\Dikte \lineskip=-\Dikte
            \halign\bgroup &\fbox{##\unskip}\unskip  \crcr }
       {\egroup\egroup\endgroup}
\def\Diagram#1{\relax\ifmmode\vcenter{\,\begin{Young}#1\end{Young}\,}\else%
              $\vcenter{\,\begin{Young}#1\end{Young}\,}$\fi}

\title[Finite $W$-algebras]{\boldmath Highest weight theory for finite $W$-algebras}
\author[J. Brundan, S. M. Goodwin and A. Kleshchev]
{Jonathan Brundan, Simon M. Goodwin and Alexander Kleshchev}

\address{Department of Mathematics, University of Oregon, Eugene, OR 97403, USA}
\email{brundan@uoregon.edu, klesh@uoregon.edu}

\address{School of Mathematics, University of Birmingham, Birmingham, B15 3LX,~UK}
\email{goodwin@for.mat.bham.ac.uk}

\thanks{
\\
\indent The first and third authors supported by the NSF grant DMS-0654147.}

\begin{document}

\begin{abstract}
We define analogues of Verma modules 
for finite $W$-algebras. By the usual ideas of
highest weight theory, this is 
a first step towards the classification
of finite dimensional irreducible modules.
We also introduce an analogue of the BGG category $\mathcal O$.
Motivated by known results in type A, 
we then formulate some precise conjectures in
the case of nilpotent orbits of standard Levi type.
\end{abstract}

\maketitle
\vspace{-2mm}
\section{Introduction}

There has been a great deal of recent interest in $W$-algebras
and their representation theory.  To each nilpotent element $e$ in 
the Lie algebra $\g$ of a complex reductive algebraic group $G$, one can
associate a finite $W$-algebra $U(\g,e)$. Up to isomorphism, this algebra depends only on the adjoint orbit $G\cdot e$ of $e$ and can be viewed informally as the
``universal enveloping algebra'' 
of the Slodowy slice to this orbit. 
Finite $W$-algebras were
introduced into the mathematical literature by Premet \cite[$\S$4]{P1};
see also \cite{GG}.  For nilpotent orbits admitting even good gradings in the
sense of \cite{EK}, these algebras already appeared  in 
work of Kostant and Lynch
\cite{Kostant, Lynch} in the context of
(generalized) Whittaker modules.
At one extreme,
$U(\g,0)=U(\g)$; at the other
extreme, when $e$ is  regular,
$U(\g,e)$ is isomorphic to the center $Z(\g)$ of $U(\g)$.

There is much motivation for studying the representation theory of finite 
$W$-algebras.  For instance, through Skryabin's equivalence
\cite{Skryabin}, there is a relationship between the
representation theory of $U(\g,e)$ and the representation theory of
$\g$. This provides an important connection between the primitive ideals
of $U(\g)$ whose associated variety contains $G \cdot e$
and the primitive ideals of $U(\g,e)$; see \cite[Theorem 3.1]{P2} and 
\cite[Theorem 1.2.2]{Losev}. 
In another direction, it is shown by
Premet \cite[$\S$6]{P1} 
that $U(\g,e)$ gives rise to a natural 
non-commutative
deformation of the singular variety that arises 
by intersecting the Slodowy slice 
to the orbit $G \cdot e$ 
with the nilpotent cone of $\g$. 

In mathematical physics,
finite $W$-algebras and their affine counterparts have attracted a
lot of attention under a slightly different
guise; see for example \cite{deBoerTjin,deVosvanDriel,Arakawa,DK}.  It has recently 
been proved in \cite{DDDHK} that
the definition in the mathematical physics literature via BRST cohomology agrees
with Premet's definition; 
see the discussion in $\S$\ref{sBRST}.

For the case $G = \mathrm{GL}_N(\C)$, the first and third authors made a
thorough study of the finite dimensional representation theory of 
$U(\g,e)$ in \cite{BK}.  In this case 
$U(\g,e)$ is isomorphic to a quotient of a 
{shifted Yangian} \cite{BKAdv}. 
Using this connection
and the natural 
triangular decomposition of shifted Yangians, we developed
a highest weight theory for $U(\g,e)$, 
leading to the classification of finite dimensional irreducible
$U(\g,e)$-modules; see \cite[$\S$7.2]{BK}.
On the other hand, 
using Premet's definition of $U(\g,e)$ and the so-called
{Whittaker functor}, we obtained character formulae
for the finite dimensional irreducibles as a  consequence of 
the Kazhdan--Lusztig conjecture for a 
certain parabolic category $\mathcal O$ attached to $\g$;
see \cite[$\S$8.5]{BK}.
This type A theory has already had several other quite striking applications;
see \cite{B3,B2,BKSchur}.


In the general case there is little concrete knowledge about 
representations of $U(\g,e)$.  It has only recently been
proved that $U(\g,e)$ always has ``enough'' finite dimensional
irreducible representations; 
see \cite[Theorem 1.1]{P3} and \cite[Theorem 1.2.3]{Losev}. 
The purpose of this paper is to set
up the framework to study representation theory of
$U(\g,e)$ via highest weight theory.  In particular, we define
Verma modules for $U(\g,e)$, which turns out to be surprisingly 
non-trivial.

Recall in classical Lie theory that 
Verma modules are ``parabolically
induced'' from irreducible representations of a Cartan
subalgebra.
The main problem for finite $W$-algebras is to find a suitable
algebra to play the role of Cartan subalgebra. It turns out that 
this role is
played by the ``smaller''
finite $W$-algebra $U(\g_0,e)$ where $\g_0$ is a minimal Levi 
subalgebra of $\g$ containing $e$, i.e. $e$ is a distinguished nilpotent
element of $\g_0$;
see $\S$\ref{sCartan}. 
Given a parametrization
$$
\{V_\La \:|\:\La \in \mathcal L\}
$$
of a complete set of pairwise inequivalent finite dimensional
irreducible $U(\g_0,e)$-modules, we will construct the Verma modules
$$
\{M(\La,e)\:|\:\La \in \mathcal L\}$$ 
for $U(\g,e)$
by parabolically inducing the $V_\La$'s from $U(\g_0,e)$ to $U(\g,e)$; 
see $\S$\ref{sVerma}.
We then prove as usual that the Verma
module $M(\La,e)$ has a unique irreducible quotient $L(\La,e)$ and that the $L(\La,e)$'s parametrized by the subset
$$
\mathcal L^+ := \{\La \in \mathcal L\:|\:\dim L(\La,e) < \infty\}
$$
give a complete set of pairwise inequivalent finite dimensional
irreducible $U(\g,e)$-modules.
Incidentally, all our Verma modules 
belong to a natural category $\mathcal O(e)$ 
 whose objects
have composition series with
only the $L(\La,e)$'s as composition factors; see $\S$\ref{sO}.
In the
case $e = 0$ this category $\mathcal O(e)$
is the usual Bernstein-Gelfand-Gelfand 
category $\mathcal O$ from \cite{BGG}.

The general principles just described reduce the problem of classifying the 
finite dimensional irreducible $U(\g,e)$-modules to 
two major problems: 
\begin{itemize}
\item[(1)] 
Find a natural parametrization of the 
finite dimensional irreducible $U(\g_0,e)$-modules by some explicit
labelling set $\mathcal L$.
\item[(2)] 
Describe the subset $\mathcal L^+$ of $\mathcal L$ combinatorially.
\end{itemize}
We remark that the definition of Verma module, hence the 
subset $\mathcal L^+$ of $\mathcal L$,
depends essentially on a choice
of positive roots in the restricted root system of $U(\g,e)$ in the
sense of \cite[$\S$2]{BG}. Unlike in the classical situation there is
often more that one conjugacy classes of such choices. The combinatorial description of the subset
$\mathcal L^+$ will certainly depend in a significant way on this choice.

In the special case that $e$ is of {\em standard Levi type}, i.e.
$e$ is actually a regular nilpotent element of $\g_0$, 
Kostant showed in \cite[$\S$2]{Kostant}
that $U(\g_0,e)$ is canonically isomorphic to $Z(\g_0)$.
Hence in this case
the solution
to problem (1) is very simple:
the set $\mathcal L$ labelling our Verma modules
can be naturally identified
with the set $\t^* / W_0$ of 
$W_0$-orbits in $\t^*$,
where $W_0$ is the Weyl group of $\g_0$ 
with respect to a maximal toral subalgebra $\t$. This resembles 
a result of Friedlander and Parshall 
\cite[Corollary 3.5]{FP} giving a similar labelling of irreducible
representations
for reduced enveloping algebras of standard Levi type in characteristic $p$.
In $\S$\ref{sLevi} we 
 formulate two explicit conjectures
concerning the standard Levi type case. The first of these conjectures
reduces the solution of problem (2) to the combinatorics of 
associated varieties of primitive ideals in $\g$.
We prove this conjecture in type A (for a standard choice of positive roots)
in $\S$\ref{sA}, by translating the results from \cite{BK} into the
general framework.
Our second conjecture is quite a bit stronger, and 
was inspired by Premet's ideas in \cite[$\S$7]{P2}.
It predicts
an explicit link between our category $\mathcal O(e)$ and another
category $\mathcal O(\chi)$ introduced by 
Mili\v c\'ic and Soergel \cite{MS}. This conjecture
also implies the truth of the
Kazhdan--Lusztig conjecture for finite $W$-algebras of standard Levi type
from \cite{deVosvanDriel}.
We speculate for $e$ of standard Levi type 
that every primitive ideal of $U(\g,e)$ is the annihilator
of an irreducible highest weight module in $\mathcal O(e)$, though we have 
no evidence for this 
beyond Duflo's theorem in the case $e=0$.

The rest of the article is organized as follows.
In Section 2, we explain in detail the
relationship between three quite different definitions of
finite $W$-algebra. The key to the
new results in this paper actually comes from the third of these definitions, namely, the BRST cohomology definition
as formulated in \cite{DK}.
We point out especially
Theorem~\ref{tidish} which makes the link between the 
second and third definitions quite transparent.
In Section 3 we survey various results of Premet describing the
associated graded algebra to $U(\g,e)$ in its two natural filtrations, setting
up more essential notation along the way. 
The main new results of the paper are proved in Section 4, the most important
being Theorem~\ref{csa}.
Finally in Section 5
we discuss standard Levi type and explain how
to translate the type A results from \cite{BK}.

We work throughout over the ground field $\C$.
By a {\em character} of a Lie algebra $\g$ we mean a Lie algebra
homomorphism $\rho:\g \rightarrow \C$.
Any such $\rho$ 
induces a shift automorphism
$$
S_\rho:U(\g) \rightarrow U(\g)$$ 
of the universal enveloping algebra $U(\g)$ with 
$S_\rho(x) := x + \rho(x)$ for each $x \in \g$.

\vspace{2mm}

\noindent
{\em Acknowledgements.}
The first author would like to thank Alexander Premet for some valuable 
discussions about finite $W$-algebras at 
the Oberwolfach meeting on enveloping algebras in March 2005.
The second author thanks the
EPSRC for the travel grant EP/F004273/1.

\section{Three definitions of finite $W$-algebras}

In this section we give three equivalent definitions of
the finite $W$-algebra $U(\g,e)$. The first two of
these
definitions have left- 
and right-handed versions which are not obviously isomorphic;
we establish that they are indeed isomorphic using the third definition.
Although not used here, we 
point out that there is also now a 
{\em fourth} important definition of the finite $W$-algebra, 
namely,
Losev's
definition via Fedosov quantization; see \cite[$\S$3]{Losev}.

\subsection{Definition via Whittaker models}\label{swhittaker}
Let $\g$ be the Lie algebra
of a connected reductive algebraic group $G$
over $\C$. Let
$e \in \g$ be a nilpotent element. By the Jacobson--Morozov theorem,
we can find $h, f \in \g$ so that $(e,h,f)$ is an
$\mathfrak{sl}_2$-triple in $\g$, i.e.\ $[h,e] = 2e, [h,f] = -2f$
and $[e,f] = h$. We write $\g^e$, $\g^f$ and $\g^h$ for the
centralizers of $e$, $f$ and $h$ in $\g$, respectively. Then $\g^h
\cap \g^e$ is a Levi factor of $\g^e$.  Pick a maximal toral
subalgebra $\t^e$ of this Levi factor, and a
maximal toral subalgebra $\t$ of $\g$ containing $\t^e$ and $h$.
So $\t^e$ is the centralizer of $e$ in $\t$.
Assume in addition that we are given a {\em good grading}
$$
\g = \bigoplus_{j \in \Z} \g(j)
$$
for $e$ that is compatible with $\t$, i.e. $e \in \g(2)$,
$\g^e \subseteq \bigoplus_{j \geq 0} \g(j)$ and $\t \subseteq \g(0)$. 
Good gradings for $e$ are classified in \cite{EK}; see
also \cite{BG}. As $h \in \t$ we have 
$h \in \g(0)$ and, by \cite[Lemma
19]{BG}, it is automatically the case that $f \in \g(-2)$. 
Any element  $x \in \g$ decomposes as $x = \sum_{j \in \Z} x(j)$
with $x(j) \in \g(j)$; we let $x(< 0) := \sum_{j < 0} x(j)$ and
$x(\geq 0) := \sum_{j \geq 0} x(j)$. From now on, we abbreviate 
$$
\p := \bigoplus_{j \geq 0} \g(j), \quad\m := \bigoplus_{j \leq -2}
\g(j), \quad\n := \bigoplus_{j < 0} \g(j), \quad\h := \g(0), \quad\k
:= \g(-1).
$$
In particular, $\p$ is a parabolic subalgebra of $\g$ with Levi
factor $\h$ and $\n$ is the nilradical of the opposite parabolic.
If the subspace $\k$ is non-zero then it is {not} a subalgebra of $\g$.
If it is zero then the good grading is necessarily
even, i.e.\ $\g(j) = \{0\}$ for all odd $j$.

Let $(.|.)$ be a 
non-degenerate symmetric invariant bilinear
form on $\g$, inducing non-degenerate forms on $\t$ and $\t^*$ in
the usual way.
Define a linear map
$$
\chi:\g \rightarrow \C, \qquad x \mapsto (e|x).
$$
Also let $\lan.|.\ran$ be the
non-degenerate symplectic form on $\k$ defined by
$$
\lan x|y \ran := \chi([y,x]).
$$
Note that $\chi$ restricts to a character of $\m$.
Let $I$ (resp.\ $\overline{I}$) be the left (resp.\ right) ideal of $U(\g)$ generated by the elements 
$
\{x - \chi(x)\:|\:x \in \m\}$. Set
$$
Q := U(\g) / I\qquad
\text{(resp.\ }\overline{Q} := U(\g) / \overline{I}\text{),}
$$
which is a left (resp.\ right) $U(\g)$-module by the regular action.
The adjoint action of $\n$ on $U(\g)$ leaves the subspace $I$ (resp.\ $\overline{I}$) invariant, so induces a 
well-defined adjoint action of
$\n$ on $Q$ (resp.\ $\overline{Q}$) such that
$$
[x, u+I] := [x,u] + I\qquad\text{(resp.\ }
[x, u+\overline{I}] := [x,u]+\overline{I}\text{)}
$$
for $x \in \n,u \in U(\g)$.
Let $Q^{\n}$ (resp.\ $\overline{Q}^{\n}$) be the corresponding invariant subspace. Then
\begin{align*}
(x-\chi(x)) (u+I) = [x,u +I]\qquad\text{(resp.\ }
(u+\overline{I}) (x-\chi(x)) =
-[x,u+\overline{I}]\text{)}
\end{align*}
for all $x
\in \m,u \in U(\g)$. 
This is all that is needed to check that the multiplication on
$U(\g)$ induces a well-defined multiplication on $Q^{\n}$ 
(resp.\ $\overline{Q}^{\n}$):
$$
(u + I)(v+I) := uv + I\qquad
\text{(resp.\ }
(u+\overline{I})(v+\overline{I}) := uv + \overline{I}\text{)}
$$
for $u+I, v +I \in Q^{\n}$ 
(resp.\ $u+\overline{I},  v+\overline{I} \in
\overline{Q}^{\n}$). 
We refer to $Q^{\n}$ as the {\em Whittaker model
realization} of the {finite $W$-algebra} associated to $e$ and the
chosen good grading. 
Up to isomorphism,
the algebra $Q^\n$ is known to be independent
of the choice of good grading; see \cite[Theorem 1]{BG} or \cite[Corollary 3.3.3]{Losev}.
Later in the section, we will
construct a canonical isomorphism between $Q^{\n}$ and the
right-handed analogue $\overline{Q}^{\n}$; the existence of
such an isomorphism is far from clear at this point.

\begin{Remark}\label{premets}\rm
The definition of $Q^\n$ just explained 
is not quite the same as Premet's
definition of the finite $W$-algebra from \cite{P1}.
To explain the connection, we need to 
fix in addition a Lagrangian subspace $\lag$ of $\k$
with respect to the form $\lan.|.\ran$.
Note that $\chi$ still restricts to a character of $\m\oplus\lag$ (though
it need not restrict to a character of $\n$).
Let $I_\lag \supseteq I$
denote the left ideal of $U(\g)$ generated by
$\{x-\chi(x)\:|\:x \in \m\oplus\lag\}$.
Set $Q_\lag := U(\g) / I_\lag$ and
\begin{align*}
Q_\lag^{\m\oplus\lag} :&= \{u+I_\lag \in Q_\lag\:|\:[x,u] \in I_\lag\text{ for all }x \in \m\oplus\lag\}\\
 &= \{u+I_\lag \in Q_\lag\:|\:(x-\chi(x))u  \in I_\lag\text{ for all }x \in \m\oplus\lag\}.
\end{align*}
Again this inherits a well-defined algebra structure from 
the multiplication in $U(\g)$; it is even the case that
$$
Q_\lag^{\m\oplus\lag} \cong \operatorname{End}_{U(\g)}(Q_\lag)^{\op}.
$$
The algebra $Q_\lag^{\m\oplus\lag}$ is exactly Premet's definition 
of the finite $W$-algebra from \cite{P1}.
By \cite[Theorem 4.1]{GG} the canonical quotient map
$Q \twoheadrightarrow Q_\lag$ 
restricts to an algebra isomorphism 
$$
\nu: Q^\n \stackrel{\sim}{\rightarrow} Q_\lag^{\m\oplus\lag}.
$$
Hence our Whittaker model realization
is equivalent to Premet's.
\end{Remark}

\subsection{Definition via non-linear Lie algebras}\label{snonlinear}
The next definition of the finite $W$-algebra is based on \cite[$\S$2.4]{P2},
and is the main definition that we will use in the subsequent sections.
To formulate it, we will use an easy special case of the 
notion of a 
non-linear Lie superalgebra
from \cite[Definition 3.1]{DK}.
For the remainder of this article, a {\em non-linear Lie superalgebra}
means a vector superspace
$\a = \a_{\0} \oplus \a_{\1}$
equipped with a {non-linear Lie bracket} $[.,.]$, that is,
 a parity preserving linear map
$\a \otimes \a \rightarrow T(\a)$
satisfying the following conditions for all homogeneous $a,b,c \in \a$:
\begin{enumerate}
\item[(1)]
$[a,b] \in \C \oplus \a$;
\item[(2)] $[a,b] = (-1)^{p(a) p(b)} [b,a]$ (where $p(a) \in \Z_2$ 
denotes parity);
\item[(3)] $[a,[b,c]] = [[a,b],c]+(-1)^{p(a) p(b)} [b,[a,c]]$ 
(interpreted using the convention that any bracket with a scalar is zero).
\end{enumerate}
This definition agrees with the general notion 
of non-linear Lie superalgebra from \cite[Definition 3.1]{DK} when the
grading on $\a$ in the general setup is concentrated in degree $1$.

The {\em universal enveloping superalgebra} of a non-linear Lie superalgebra
$\a$ is 
$U(\a) := T(\a)/M(\a)$ where $M(\a)$ is the two-sided ideal generated
by the elements $a \otimes b - (-1)^{p(a) p(b)} b \otimes a - [a,b]$
for all homogeneous
$a,b \in \a$.
By a special case of 
\cite[Theorem 3.3]{DK}, $U(\a)$ is {\em PBW generated} by $\a$ in the sense 
that if
$\{x_i\:|\:i \in I\}$ is any homogeneous ordered basis of $\a$ 
then the ordered monomials
$$
\{x_{i_1} \cdots x_{i_s}\:|\:s \geq 0, i_1 \leq \cdots \leq i_s
\text{ and }i_t < i_{t+1} \text{ if }p(x_{i_t}) = \1\}
$$
give a
basis for $U(\a)$.
By a {\em subalgebra} of a non-linear Lie superalgebra $\a$
we mean a graded subspace $\b$ of $\a$ such that 
$[\b,\b] \subseteq \C \oplus \b$.
In that case $\b$ is itself a non-linear Lie superalgebra
and $U(\b)$ 
is identified with the subalgebra
of $U(\a)$ generated by $\b$.
We call $\a$ a {\em non-linear Lie algebra}
if it is purely even.

Now return to the setup of $\S$\ref{swhittaker}.
Following the language of \cite[\S
5]{DK} and \cite{DDDHK}, let
$$
\k^{\ne} = \{x^{\ne}\:|\:x \in \k\}
$$
be a ``neutral'' copy of $\k$. We allow ourselves to write $x^{\ne}$
for any element $x \in \g$, meaning $x(-1)^{\ne}$. Make
$\k^{\ne}$ into a non-linear Lie algebra with non-linear
Lie bracket defined by
$$
[x^\ne,y^\ne] := \lan x | y \ran
$$
for $x, y \in \k$,
recalling that $\lan x| y \ran = (e|[y,x])$.
Then $U(\k^\ne)$
is the Weyl algebra associated to 
$\k$ and the symplectic form 
$\lan.|.\ran$. Let
$$
\widetilde{\g} := \g \oplus \k^{\ne}
$$
viewed as a non-linear Lie algebra with bracket obtained by
extending the brackets already defined on $\g$ and $\k^\ne$ to all
of $\widetilde{\g}$ by declaring $[x, y^{\ne}] := 0$ for $x \in
\g, y \in \k$. Then
$U(\widetilde{\g})=
U(\g) \otimes U(\k^{\ne})$. Also
introduce the subalgebra
$$
\widetilde{\p} := \p \oplus
 \k^{\ne}
$$
of $\widetilde{\g}$, whose universal enveloping algebra is
identified with $U(\p) \otimes U(\k^{\ne})$.
For use in $\S$\ref{sKazhdan},
we record the following crucial lemma which
is proved as in \cite[(2.2)]{GG}.

\begin{Lemma}\label{wpd}
$\displaystyle\widetilde{\p} = \g^e \oplus \bigoplus_{j \geq 2} [f,
\g(j)] \oplus \k^\ne$.
\end{Lemma}

Extend the left (resp.\ right) regular action of $\g$ on $Q$ (resp.\ $\overline{Q}$)
to an action of $\widetilde{\g}$ by setting
$$
x^\ne (u +I) := u x +I\qquad\text{(resp.\ }
(u+\overline{I})x^\ne := xu+\overline{I}\text{)}
$$
for $u \in U(\g)$ and $x \in \k$. This makes $Q$ into a left
$U(\widetilde{\g})$-module (resp.\ $\overline{Q}$ into a right
$U(\widetilde{\g})$-module). For $x,y \in \n$, we have that 
$$
[x-\chi(x) - x^\ne, y - \chi(y) - y^\ne] = [x,y] - \chi([x,y]) -
[x,y]^\ne,
$$
because $[x,y]^\ne = 0$.
Hence the map $\n \rightarrow
U(\widetilde{\g}), \:x \mapsto x - \chi(x) - x^\ne$ is a Lie algebra
homomorphism. So we can make $U(\widetilde{\g})$ into
an $\n$-module via the {\em twisted adjoint action} defined by
letting $x \in \n$ act as the derivation $u \mapsto [x-\chi(x) -
x^\ne,u]$. Since $\chi(x)$ is a scalar, this map can be written more
succinctly as $u \mapsto [x-x^\ne, u]$.

\begin{Lemma}\label{nat}
The natural multiplication map
$$
U(\widetilde{\g}) \twoheadrightarrow Q, \ u \mapsto u(1 + I)
\qquad\text{(resp.\ }
U(\widetilde{\g}) \twoheadrightarrow \overline{Q}, \  u \mapsto
(1+\overline{I})u\text{)}
$$
intertwines the twisted adjoint action of $\n$ on $U(\widetilde{\g})$
with the adjoint action of $\n$ on $Q$ (resp.\ $\overline{Q}$).
\end{Lemma}

\begin{proof}
We verify this in the left-handed case, the other case being
similar.
We need to show that
$[x-x^\ne,u](1 +I) = [x,u(1+I)]$ for $x \in \n$ and $u \in U(\widetilde{\g})$.
We may assume that $u = v y_1^\ne \cdots y_n^\ne$ for
$v \in U(\g)$ and $y_1,\dots,y_n \in \k$. Then,
$$
[x-x^\ne,u]  = [x,v] y_1^\ne\cdots y_n^\ne
- v [x^\ne, y_1^\ne \cdots y_n^\ne].
$$
Acting on $1+I$, we get that
\begin{align*}
[x-x^\ne,u](1+I)&=
[x,v] y_n \cdots y_1
+ v [x,y_n \cdots y_1] +I\\
&= [x, v y_n \cdots y_1] +I
= [x,vy_n\cdots y_1 +I]
= [x,u(1 +I)]
\end{align*}
as required.
\end{proof}

Let $J$ (resp.\ $\overline{J}$) be the left (resp.\ right) ideal of $U(\widetilde{\g})$ generated by the elements 
$
\{x - \chi(x) - x^{\ne}\:|\:x \in \n\}.
$
By the PBW theorem, we have that
\begin{equation*}
 U(\widetilde{\g}) =
U(\widetilde{\p}) \oplus J
\qquad\text{(resp.\ }
 U(\widetilde{\g})
= U(\widetilde{\p}) \oplus \overline{J}\text{)}.
\end{equation*}
Let
$\Pr:U(\widetilde{\g}) \rightarrow U(\widetilde{\p})$ (resp.\ $\overline{\Pr}:U(\widetilde{\g}) \rightarrow U(\widetilde{\p})$)
denote the corresponding linear projection.
Define \begin{align*}
U(\g,e) &:= \{u \in U(\widetilde{\p})\:|\:
\Pr([x-x^\ne,u]) = 0 \text{ for all }x \in \n\},\\
\overline{U}(\g,e) &:= \{u \in U(\widetilde{\p})\:|\:
\overline{\Pr}([x-x^\ne,u]) = 0 \text{ for all }x \in \n\}.
\end{align*}

\begin{Theorem}\label{t1}
The subspaces $U(\g,e)$ 
and $\overline{U}(\g,e)$ are subalgebras of
$U(\widetilde{\p})$, and the maps
$$
U(\g,e) \rightarrow Q^{\n},
 \ u \mapsto u (1+I),\qquad
\overline{U}(\g,e) \rightarrow \overline{Q}^{\n}, \ u \mapsto
(1+\overline{I}) u
$$
are well-defined algebra isomorphisms.
\end{Theorem}

\begin{proof}
We deal with the left-handed case, the
right-handed analogue being similar. Note for $x \in \n$ that
$(x-\chi(x)-x^\ne) (1+I) \subseteq I$, so $J (1+I) \subseteq I$. 
Hence $\Pr(u) (1+I) = u  (1+I)$ for every $u \in U(\widetilde{\g})$. Given this, Lemma~\ref{nat} implies that
$$
\Pr([x-x^\ne,u]) (1+I) =
[x-x^\ne, u] (1+I) = [x,u (1+I)]
$$
for $x \in \n$ and $u \in U(\widetilde{\p})$. 
By the PBW theorem,
the map $U(\widetilde{\p}) \rightarrow Q,\:u
\mapsto u(1+I)$ is a vector space isomorphism. Putting our
observations together, $u \in U(\widetilde{\p})$ belongs
to $U(\g,e)$ if and only if $u (1+I)$ belongs to $Q^{\n}$. Thus the 
map $U(\widetilde{\p})\stackrel{\sim}{\rightarrow} Q$ restricts to a
vector space isomorphism $U(\g,e) \stackrel{\sim}{\rightarrow}
Q^{\n}$.

Now take general elements $\sum_x u_x x_1^\ne \cdots x_s^\ne$ and
$\sum_y v_y y_1^\ne \cdots y_t^\ne$ in $U(\g,e)$ for $u_x, v_y \in
U(\p)$ and $x_1,\dots,x_s, y_1,\dots,y_t\in \k$. The image in $Q$ of
their product in $U(\widetilde{\p})$ is equal to
$$
\textstyle\sum_{x,y} u_x v_y y_t \cdots y_1 x_s \cdots x_1 +I.
$$
We claim that this is equal to the product of their images in
$Q^{\n}$, namely,
$$
\textstyle\sum_{x,y} u_x x_s \cdots x_1 v_y y_t \cdots y_1 + I.
$$
To see this, note that $\sum_y v_y y_t \cdots y_1 +I$ belongs to
$Q^{\n}$. Hence for each $i$ the commutator $[x_i, \sum_y v_y y_t
\cdots y_1]$ belongs to $I$. Also observe that $I x_{i-1} \cdots x_1
\subseteq I$. Using these two facts applied successively to $i=1,\dots,s$,
we get that
$$
x_s \cdots x_1 \sum_y v_y y_t\cdots y_1 +I
= \sum_y v_y y_t \cdots y_1 x_s \cdots x_1 +I
$$
and the claim follows.

The claim shows that the product of two elements in $U(\g,e)$ is again
an element of $U(\g,e)$, because the image of the product lies in
$Q^{\n}$. Hence $U(\g,e)$ is indeed a subalgebra of
$U(\widetilde{\p})$. At the same time, the claim establishes that
our vector space isomorphism is actually an algebra isomorphism.
\end{proof}

We refer to $U(\g,e)$ simply as
the {\em finite $W$-algebra} associated to $e$.

\subsection{Definition via BRST cohomology}\label{sBRST}
We now turn to the third 
definition of the 
$W$-algebra.
This has 
been proved to be equivalent to the Whittaker
model definition above 
in \cite{DDDHK}. Let
$$
\n^\ch = \{x^\ch\:|\:x \in \n\}
$$
be a ``charged'' copy of $\n$. As before, we allow ourselves to
write $x^\ch$ for any $x \in \g$, meaning $x(< 0)^\ch$. Recalling
that $\widetilde{\g} = \g \oplus \k^\ne$, let
$$
\widehat{\g} := \widetilde{\g} \oplus \n^* \oplus \n^\ch
$$
viewed as a 
non-linear Lie superalgebra with even part equal to
$\widetilde{\g}$, odd part equal to $\n^* \oplus \n^\ch$, and
non-linear Lie bracket $[.,.]$ defined as follows. It is equal
to the non-linear Lie bracket defined above on $\widetilde{\g}$. It
is identically zero on $\n^*$, $\n^\ch$ or between elements of
$\widetilde{\g}$ and $\n^* \oplus \n^\ch$. Finally $[f, x^\ch] :=
\lan f,x \ran$ for $f \in \n^*, x \in \n$, where $\lan f,x \ran$
denotes the natural pairing of $f \in \n^*$ with $x \in \n$. We also
have the subalgebra
$$
\widehat{\p} := \widetilde{\p} \oplus \n^* \oplus \n^\ch
$$
of $\widehat{\g}$. We put 
the {\em cohomological
grading} on $\widehat{\g}$,
hence also on $\widehat{\p}$, consistent with the $\Z_2$-grading, by
declaring that elements of $\widetilde{\g}$ are in degree $0$,
elements of $\n^*$ are in degree $1$, and elements of $\n^\ch$ are
in degree $-1$.
It induces gradings 
$$
U(\widehat{\g}) = \bigoplus_{i \in \Z} U(\widehat{\g})^i, \qquad
U(\widehat{\p}) = \bigoplus_{i \in \Z} U(\widehat{\p})^i.
$$

Fix a basis
$b_1,\dots,b_r$ for $\n$ such that $b_i$ 
lies in the $\be_i$-weight space of $\g(-d_i)$ 
for some $\be_i \in \t^*$
and $d_i > 0$. Define the structure constants $\gamma_{i,j,k} \in
\C$ from
$$
[b_i, b_j] = \sum_{k=1}^r \gamma_{i,j,k} b_k.
$$
Let $f_1,\dots,f_r$ be the dual basis for $\n^*$. The coadjoint
action of $\n$ on $\n^*$ defined by $\lan (\coad b) ( f), b'\ran =
-\lan f, (\ad b)(b')\ran$ for $b, b' \in \n$ and $f \in \n^*$
satisfies
$$
(\coad b_i)(f_j) = \sum_{k=1}^r \gamma_{k,i,j} f_k.
$$
Let $d: U(\widehat{\g}) \rightarrow U(\widehat{\g})$ be the
superderivation of cohomological degree $1$ defined by taking the
supercommutator with the odd element
$$
\sum_{i=1}^r f_i (b_i - \chi(b_i) - b_i^\ne) - {\textstyle\frac{1}{2}}
\sum_{i,j=1}^r f_i f_j [b_i, b_j]^\ch.
$$
As in \cite{DDDHK}, one checks:
\begin{align*}
d(x) &= \sum_{i=1}^r f_i [b_i, x]&(x \in \g),\\
d(f) &= {\textstyle\frac{1}{2}}\sum_{i=1}^r f_i \,(\coad b_i) (f)&(f \in \n^*),\\
d(x^\ch) &= x - \chi(x) - x^\ne + \sum_{i=1}^r f_i [b_i, x]^\ch&(x \in \n),\\
d(x^\ne) &= \sum_{i=1}^r f_i \chi([b_i, x])&(x \in \k).
\end{align*}
Using these formulae it is easy to check that $d^2 = 0$, i.e.\
$(U(\widehat{\g}), d)$ is 
a differential graded superalgebra.
Let $H^\bullet(U(\widehat{\g}), d)$ be its 
cohomology. It is
known from \cite{DDDHK} that this is concentrated in
cohomological degree $0$. 
So 
$$
H^\bullet(U(\widehat{\g}), d) = \ker(d:U(\widehat{\g})^0 \rightarrow
U(\widehat{\g})^1) \:\big / \: \im(d:U(\widehat{\g})^{-1}
\rightarrow U(\widehat{\g})^0).
$$
Note by the PBW theorem that
$$
U(\widehat{\g})^0 = U(\widetilde{\g}) \oplus
\n^* U(\widehat{\g})^0 \n^\ch
\qquad
\text{(resp.\ }
U(\widehat{\g})^0 = U(\widetilde{\g}) \oplus \n^\ch
U(\widehat{\g})^0 \n^*\text{)},
$$
with $\n^* U(\widehat{\g})^0 \n^\ch$ (resp.\ $\n^\ch U(\widehat{\g})^0
\n^*$) being a two-sided ideal. So we can define 
a linear map
$$
q:U(\widehat{\g})^0 \twoheadrightarrow Q
\qquad
\text{(resp.\ }
\overline{q}:U(\widehat{\g})^0 \twoheadrightarrow \overline{Q}\text{)}
$$
such that
$q(u) = u (1+I)$ 
(resp.\ $\overline{q}(u) = (1+\overline{I}) u$)
 for
$u \in U(\widetilde{\g})$
and
$$
\ker q = J \oplus
\n^* U(\widehat{\g})^0 \n^\ch
\qquad
\text{(resp.\ }
\ker \overline{q} = \overline{J} \oplus \n^\ch U(\widehat{\g})^0
\n^*\text{)}.
$$
By the above explicit formulae for the differential $d$, it follows
that $d$ maps $U(\widehat{\g})^{-1}$ into $J \oplus \n^*
U(\widehat{\g})^0 \n^\ch$ (resp.\ $\overline{J} \oplus \n^\ch
U(\widehat{\g})^0 \n^*$). So the restriction of $q$ (resp.\ $\overline{q}$) to $\ker(d:U(\widehat{\g}^{0}) \rightarrow
U(\widehat{\g})^1)$ induces a well-defined linear map
$$
q:H^\bullet(U(\widehat{\g}), d) \rightarrow Q
\qquad\text{(resp.\ }
\overline{q}:H^\bullet(U(\widehat{\g}), d) \rightarrow
\overline{Q}\text{)}.
$$
In \cite{DDDHK}, it is proved that $q$
is
an {\em algebra isomorphism} between
$H^\bullet(U(\widehat{\g}), d)$ and $Q^{\n}$. A 
similar
argument shows that $\overline{q}$
is an algebra
isomorphism between $H^\bullet(U(\widehat{\g}), d)$ and
$\overline{Q}^{\n}$. This already shows that
$
Q^{\n} \cong \overline{Q}^{\n}
$
as algebras,
though it does not give the isomorphism
as explicitly as we would like.

To remedy this, note that
the projection along the 
decomposition
$$
U(\widehat{\p})^0 = U(\widetilde{\p}) \oplus
\n^* U(\widehat{\p})^0 \n^\ch
\qquad
\text{(resp.\ }
U(\widehat{\p})^0 = U(\widetilde{\p}) \oplus \n^\ch
U(\widehat{\p})^0 \n^*\text{)}
$$
defines a
surjective algebra homomorphism
$$
p:U(\widehat{\p})^0 \twoheadrightarrow U(\widetilde{\p})
\qquad
\text{(resp.\ }
\overline{p}:U(\widehat{\p})^0 \twoheadrightarrow U(\widetilde{\p})
\text{).}
$$
Moreover the following diagrams commute:
$$
\begin{CD}
U(\widehat{\p})^0&@>>>&U(\widehat{\g})^0\\
@Vp VV&&@V Vq V\\
U(\widetilde{\p}) &@>\sim>>&Q
\end{CD}
\qquad\qquad
\begin{CD}
U(\widehat{\p})^0&@>>>&U(\widehat{\g})^0\\
@V\overline{p} VV&&@VV \overline{q} V\\
U(\widetilde{\p}) &@>\sim>>&\overline{Q}
\end{CD}
$$
where the top maps are the inclusions and the bottom maps are the
multiplication maps defined like in Lemma~\ref{nat}. For the next
lemma, recall that the basis element $b_i \in \n$ is of $\t$-weight $\be_i$.

\begin{Lemma}\label{fine}
$\be := \sum_{i=1}^r \be_i \in \t^*$ extends uniquely
to a character of $\p$.
\end{Lemma}

\begin{proof}
For $x \in \h$, the linear map $\ad x: \g \rightarrow \g$ leaves
$\n$ invariant, so the map
$$
\be: \h \rightarrow \C, \qquad x\mapsto \tr(\ad x|_{\n})
$$
is a well-defined Lie algebra homomorphism. Extending by zero on the
nilradical of $\p$ we get the desired Lie algebra homomorphism
$\be: \p \rightarrow \C$. The uniqueness is clear.
\end{proof}

In view of the lemma, we can define shift automorphisms
$$
S_{\pm\be}:U(\widetilde{\p}) \rightarrow U(\widetilde{\p}), \quad x\mapsto x \pm \be(x),\ y^\ne \mapsto y^\ne\qquad(x \in \p, y \in \k).
$$
Of course, we have
that $(S_{\pm\be})^{-1} = S_{\mp\be}$.
The definition
 of the function $\phi$ 
in the next lemma is based on
a construction of Arakawa \cite[2.5]{Arakawa} for regular nilpotents.

\begin{Lemma}\label{tidy}
There are well-defined algebra homomorphisms
$$
\phi: U(\widetilde{\p}) \hookrightarrow
U(\widehat{\p})^0,\qquad\qquad
\overline{\phi}: U(\widetilde{\p}) \hookrightarrow
U(\widehat{\p})^0
$$
such that
\begin{align*}
\phi(x) &= x + \textstyle\sum_{i} f_i [b_i, x]^\ch,
&
\overline{\phi}(x) &= x - \textstyle\sum_{i} [b_i, x]^\ch f_i,\\
\phi(y^\ne) &= y^\ne,
&\overline{\phi}(y^\ne) &= y^\ne,
\end{align*}
for $x \in \p$ and $y \in \k$. Moreover,
\begin{itemize}
\item[(1)]
$p \circ \phi = \overline{p} \circ \overline{\phi} =
\id_{U(\widetilde{\p})}$;
\item[(2)]
$\overline{\phi} = \phi \circ S_\be$.
\end{itemize}
Hence $\phi$ and $\overline{\phi}$ are injective,
$p \circ \overline{\phi} = S_\be$
and \,$\overline{p} \circ \phi = S_{-\be}$.
\end{Lemma}

\begin{proof}
To see that there is an
algebra homomorphism
$\phi:U(\widetilde{\p}) \rightarrow U(\widehat{\p})^0$ defined on
generators as in the statement of the lemma, 
we need to
show that $\phi(x) \phi(y) - \phi(y) \phi(x) = \phi([x,y])$ for all
$x, y \in \widetilde{\p}$. 
If $x, y\in\k^\ne$, the result is clear. 
If $x \in \p$ and $y \in \k^\ne$ or vice
versa, both sides are obviously zero. It remains to consider the case that $x, y \in
\p$.  Then
\begin{align*}
\phi(x)\phi(y) &=
xy + \textstyle\sum_{i} x f_i [b_i, y]^\ch + \textstyle\sum_{i} f_i [b_i, x]^\ch y
+ \textstyle\sum_{i,j} f_i [b_i, x]^\ch f_j [b_j,y]^\ch,\\
\phi(y)\phi(x) &=
yx + \textstyle\sum_{i} y f_i [b_i, x]^\ch + \textstyle\sum_{i} f_i [b_i, y]^\ch x
+ \textstyle\sum_{i,j} f_j [b_j, y]^\ch f_i [b_i,x]^\ch.
\end{align*}
Now note that
\begin{align*}
\textstyle\sum_{i,j} f_i [b_i,x]^\ch &f_j [b_j,y]^\ch\\
&=
-\textstyle\sum_{i,j} f_i f_j [b_i, x]^\ch [b_j,y]^\ch
+
\textstyle\sum_{i,j}f_i \lan f_j, [b_i,x](< 0)\ran  [b_j,y]^\ch\\
&=
-\textstyle\sum_{i,j} f_i f_j [b_i, x]^\ch [b_j,y]^\ch
+
\textstyle\sum_{i} f_i [[b_i,x](< 0),y]^\ch\\
&=
-\textstyle\sum_{i,j} f_i f_j [b_i, x]^\ch [b_j,y]^\ch
+
\textstyle\sum_{i} f_i [[b_i,x],y]^\ch.
\end{align*}
Similarly,
\begin{align*}
\textstyle\sum_{i,j} f_j [b_j,y]^\ch f_i [b_i,x]^\ch
&=
-\textstyle\sum_{i,j} f_j f_i [b_j, y]^\ch [b_i,x]^\ch
+
\textstyle\sum_{j} f_j [[b_j,y],x]^\ch\\
&=-\textstyle\sum_{i,j} f_i f_j  [b_i,x]^\ch [b_j, y]^\ch
-
\textstyle\sum_{i} f_i [x,[b_i,y]]^\ch.
\end{align*}
Thus, we finally get the required equality:
\begin{align*}
\phi(x)\phi(y)
-
\phi(y)\phi(x)
&= [x,y] + \textstyle\sum_{i} f_i \left([[b_i,x],y]^\ch + [x, [b_i,y]]^\ch\right) \\
&= [x,y] + \textstyle\sum_{i} f_i [b_i,[x,y]]^\ch =
\phi([x,y]).
\end{align*}

Now {\em define} an algebra homomorphism $\overline{\phi} := \phi \circ S_\be$.
For $y\in\k$ and $x \in \p$
we have that 
\begin{align*}
\overline{\phi}(y^\ne)& = \phi(S_\be(y^\ne)) = y^\ne,
\\
\overline{\phi}(x) &= \phi(S_\be(x)) = \phi(x + \be(x))
=
x+\be(x) + \textstyle\sum_{i} f_i [b_i, x]^\ch\\
&= x + \be(x) + \textstyle\sum_{i} \lan f_i,[b_i,x](< 0)\ran -
\sum_{i} [b_i,x]^\ch f_i.
\end{align*}
If $x$ belongs to the nilradical of $\p$, then $\lan f_i, [b_i,x](<
0)\ran = 0 = - \be(x)$ by degree considerations. Instead if $x
\in \h$ then
$$
\textstyle\sum_{i} \lan f_i ,[b_i,x](< 0)\ran = -\sum_{i} \lan
f_i,(\ad x)(b_i)\ran = - \tr (\ad x|_{\n}) = - \be(x),
$$
recalling the proof of Lemma~\ref{fine}.
Hence
$
\overline{\phi}(x) = x - \textstyle\sum_{i} [b_i, x]^\ch f_i
$
as in the statement of the lemma.
Finally the property (1) is obvious.
\end{proof}

\begin{Lemma}\label{mess}
For $u \in U(\widetilde{\p})$, we have that
$$
d(\phi(u)) = \textstyle\sum_{i} f_i \phi(\Pr ([b_i - b_i^\ne, u])),
\ \ d(\overline{\phi}(u)) = \textstyle\sum_{i} \overline{\phi}(\overline{\Pr}
([b_i - b_i^\ne, u]))f_i.
$$
\end{Lemma}

\begin{proof}
We deal with the left-handed version, the argument
for the right-handed analogue being 
similar. We first check
the result for any $u \in U(\k^\ne)$ by induction on the
natural filtration. The base case is when $u$ is a
scalar, which is trivial as both sides are zero. For the induction
step, take $u = y^{\ne} v$ for $y \in \k$ and $v \in U(\k^\ne)$.
Since $[b_i,v] = 0$, the induction hypothesis gives 
$$
d(\phi(v)) =
\textstyle\sum_{i} f_i \phi(\Pr([b_i-b_i^\ne,v]))
=\textstyle\sum_{i} f_i \phi(\Pr([-b_i^\ne,v])).
$$
Hence,
\begin{align*}
d(\phi(y^\ne v)) &= d(\phi(y^\ne )\phi(v)) = d(\phi(y^\ne))\phi(v) + \phi(y^\ne)
d(\phi(v))\\
&= \textstyle\sum_{i} f_i \chi([b_i,y]) \phi(v) +
\phi(y^\ne)
\textstyle\sum_{i} f_i \phi(\Pr([-b_i^\ne,v]))\\
&= \textstyle\sum_{i} f_i \phi\big(\chi([b_i,y]) v +
y^\ne\Pr([-b_i^\ne,v])\big)\\
&= \textstyle\sum_{i} f_i \phi\big([-b_i^\ne,y^\ne]v +
\Pr(y^\ne[-b_i^\ne,v])\big)
\end{align*}\begin{align*}
\phantom{d(\phi(y^\ne v))}
&= \textstyle\sum_{i} f_i \phi\big(\Pr([-b_i^\ne,y^\ne]v +
y^\ne[-b_i^\ne,v])\big)\\&=
\textstyle\sum_{i} f_i \phi\big(\Pr([-b_i^\ne,y^\ne v])\big)=\textstyle\sum_{i} f_i \phi\big(\Pr([b_i-b_i^\ne,y^\ne v])\big)
\end{align*}
as we wanted.

Next we prove the result for $x \in \p$, when
\begin{align*}
d(\phi(x)) &= d(x)+
\textstyle\sum_{i} d(f_i[b_i,x]^\ch)\\
&= \textstyle\sum_{i}f_i[b_i,x]+
\frac{1}{2}\textstyle\sum_{i,j} f_j (\coad b_j)(f_i) [b_i,x]^\ch
- \textstyle\sum_{i} f_i [b_i,x](< 0)\\&\qquad +\textstyle\sum_{i} f_i  \chi([b_i,x]) +\textstyle\sum_{i} f_i  [b_i,x]^\ne
- \textstyle\sum_{i,j} f_i f_j [b_j,[b_i,x](< 0)]^\ch\\
&= \textstyle\sum_{i}f_i[b_i,x](\geq 0)
+\textstyle\sum_{i} f_i  \chi([b_i,x]) +\textstyle\sum_{i} f_i  [b_i,x]^\ne\\
&\qquad
+\textstyle\frac{1}{2}\sum_{i,j} f_j (\coad b_j)(f_i) [b_i,x]^\ch
- \textstyle\sum_{i,j} f_i f_j [b_j,[b_i,x](< 0)]^\ch
\\
&= \textstyle\sum_{i}f_i\Pr([b_i,x])
+\textstyle\frac{1}{2}\sum_{i,j,k} \gamma_{k,j,i} f_j f_k
 [b_i,x]^\ch
\\&\qquad- \textstyle\sum_{i,j} f_i f_j [b_j,[b_i,x](< 0)]^\ch\\
\phantom{d(\phi(x))}
&=
\textstyle\sum_{i}f_i\phi(\Pr([b_i,x]))
+\frac{1}{2}\textstyle\sum_{i,j,k} \gamma_{k,j,i} f_j f_k
 [b_i,x]^\ch\\
&\qquad
-\textstyle\sum_{i,j}f_if_j[b_j,\Pr([b_i,x])]^\ch
- \textstyle\sum_{i,j} f_i f_j [b_j,[b_i,x](< 0)]^\ch
\\
&=
\textstyle\sum_{i}f_i\phi(\Pr([b_i,x]))
+\frac{1}{2}\textstyle\sum_{i,j,k} \gamma_{k,j,i} f_j f_k
 [b_i,x]^\ch\\
&\qquad
-\textstyle\sum_{i,j}f_if_j[b_j,[b_i,x](\geq 0)]^\ch
- \textstyle\sum_{i,j} f_i f_j [b_j,[b_i,x](< 0)]^\ch\\
&=
\textstyle\sum_{i}f_i\phi(\Pr([b_i,x]))
+\frac{1}{2}\textstyle\sum_{i,j,k} \gamma_{j,i,k} f_i f_j
 [b_k,x]^\ch\\
&\qquad -
\textstyle\sum_{i,j}f_if_j[b_j,[b_i,x]]^\ch.
\end{align*}
Finally we observe that
\begin{align*}
\textstyle\sum_{i,j}f_if_j[b_j,[b_i,x]]^\ch &=
\textstyle\sum_{i,j}f_if_j[b_i,[b_j,x]]^\ch
+ \textstyle\sum_{i,j}f_if_j[[b_j,b_i],x]^\ch\\
&=
-\textstyle\sum_{i,j}f_if_j[b_j,[b_i,x]]^\ch
+
\textstyle\sum_{i,j,k}\gamma_{j,i,k}f_if_j[b_k,x]^\ch.
\end{align*}
Hence
$$
\textstyle\sum_{i,j}f_if_j[b_j,[b_i,x]]^\ch =
\frac{1}{2}
\textstyle\sum_{i,j,k}\gamma_{j,i,k}f_if_j[b_k,x]^\ch
$$
and we have proved that
$$
d(\phi(x)) =
\textstyle\sum_{i} f_i \phi(\Pr([b_i,x]))
=\textstyle\sum_{i} f_i \phi(\Pr([b_i-b_i^\ne,x]))
$$
as required.

To finish the proof, we use induction on the standard filtration on
$U(\widetilde{\g})$. Take any $x \in \p$ and $u \in
U(\widetilde{\p})$. Note that
$[\phi(x), f_i] = \textstyle\sum_{j} f_j \lan f_i, [b_j,x](< 0) \ran$.
Using this and the induction hypothesis, we get that
\begin{align*}
d(\phi(xu)) &=
d(\phi(x) \phi(u)) =
d(\phi(x)) \phi(u) + \phi(x) d(\phi(u))\\
&= \textstyle\sum_{i} f_i \phi(\Pr([b_i-b_i^\ne,x])) \phi(u) +
\textstyle\sum_{i} \phi(x) f_i \phi(\Pr([b_i-b_i^\ne,u]))\\
&= \textstyle\sum_{i} f_i \phi(
\Pr([b_i,x]) u) + \sum_i f_i \phi(x\Pr([b_i-b_i^\ne,u]))
\\&\qquad+ \textstyle\sum_{i,j} f_i  \phi\big(\lan f_j, [b_i,x](< 0) \ran\Pr([b_j-b_j^\ne,u])\big)\\
&= \textstyle\sum_{i} f_i \phi\big
(\Pr([b_i,x]) u + x\Pr([b_i-b_i^\ne,u])
\\&\qquad
+ \Pr([[b_i,x](< 0),u])
- \Pr([[b_i,x]^\ne,u])\big).
\end{align*}
We are trying to show that this equals
$$
\textstyle\sum_i f_i \phi\big(\Pr([b_i-b_i^\ne,xu])\big)
=
\sum_i f_i \phi\big(\Pr([b_i,x] u)
+ \Pr(x [b_i-b_i^\ne,u])\big).
$$
So it remains to show  that
$$
\Pr([b_i,x] u)
-\Pr([b_i,x]) u
- \Pr([[b_i,x](< 0),u])
+ \Pr([[b_i,x]^\ne,u]) = 0.
$$
To see this, we expand each term separately:
\begin{align*}
\Pr([b_i,x] u) &= [b_i,x](\geq 0) u + \Pr([b_i,x](< 0) u),\\
-\Pr([b_i,x]) u &= -[b_i,x](\geq 0) u - [b_i,x]^\ne u - \chi([b_i,x])u,\\
-\Pr([[b_i,x](< 0), u]) &=
-\Pr([b_i,x](< 0) u) + u [b_i,x]^\ne + u \chi([b_i,x]),\\
\Pr([[b_i,x]^\ne, u]) &=
[b_i,x]^\ne u - u [b_i,x]^\ne.
\end{align*}
Adding these
together gives zero, completing the
induction step.
\end{proof}

\begin{Theorem}\label{tidish}
We have that
$$
U(\g,e) = \{u \in U(\widetilde{\p})\:|\: d(\phi(u)) = 0\},\quad
\overline{U}(\g,e) = \{u \in U(\widetilde{\p})\:|\:
d(\overline{\phi}(u)) = 0\}.
$$
Moreover, we have that $\ker d = \phi(U(\g,e)) \oplus \im\, d =
\overline{\phi}(\overline{U}(\g,e)) \oplus \im\, d$.
\end{Theorem}

\begin{proof}
As usual, we just prove the left-handed version.
By definition,
$$
U(\g,e) = \{u \in U(\widetilde{\p})\:|\:\Pr([b_i - b_i^\ne, u]) = 0
\text{ for each }i=1,\dots,r\}.
$$
The injectivity of $\phi$, implies that $\Pr([b_i - b_i^\ne,u]) = 0$
for each $i=1,\dots,r$ if and only if $\sum_{i=1}^r f_i
\phi(\Pr([b_i-b_i^\ne,u])) = 0$. By Lemma~\ref{mess}, this is
precisely the statement that $d(\phi(u)) = 0$. In particular, this
shows that $\phi(U(\g,e)) \subseteq \ker d$. Finally, recalling the
commutative diagram immediately before Lemma~\ref{fine}, we consider
the induced commutative diagram
$$
\begin{CD}
\phi(U(\g,e))&@>>>&H^\bullet(U(\widehat{\g}), d) \\
@Vp VV&&@VVq V\\
U(\g,e)&@>\sim >>& Q^{\n}
\end{CD}
$$
where the top map is the map $u \mapsto u + \ker d$. The left hand
map is an isomorphism by Lemma~\ref{tidy}. We have already observed,
the right hand map is an isomorphism by \cite{DDDHK}. Hence the top
map is an isomorphism too, showing that $\ker d = \phi(U(\g,e)) \oplus
\im \,d$.
\end{proof}

We are now in a position to  give the promised 
explicit isomorphism
between $U(\g,e)$ and $\overline{U}(\g,e)$.

\begin{Corollary}\label{rhs}
The restrictions of the automorphisms $S_{\pm\be}$ of
$U(\widetilde{\p})$ define mutually inverse algebra isomorphisms
$$
S_\be: \overline{U}(\g,e) \stackrel{\sim}{\rightarrow} U(\g,e),
\qquad
S_{-\be}: U(\g,e) \stackrel{\sim}{\rightarrow} \overline{U}(\g,e).
$$
\end{Corollary}

\begin{proof}
By Theorem~\ref{tidish}, 
$
d(\phi(U(\g,e))) = \{0\} \text{ and }
d(\overline{\phi}(\overline{U}(\g,e))) = \{0\}.
$
Thus by Lemma~\ref{tidy},
$
d(\phi(S_\be(\overline{U}(\g,e)))) = \{0\}.
$
So by Theorem~\ref{tidish} again, we have 
$S_\be(\overline{U}(\g,e)) \subseteq U(\g,e)$. Similarly, we have
$S_{-\be}(U(\g,e)) \subseteq \overline{U}(\g,e)$.
\end{proof}

Using Corollary~\ref{rhs} it is an easy matter to 
translate statements
about $U(\g,e)$ into analogous statements about 
$\overline{U}(\g,e)$.
So for the remainder of the article we will just formulate
things in the left-handed case.

\section{Associated graded algebras}

Finite $W$-algebras possess two important filtrations. In this
section we review the fundamental
theorems describing the corresponding 
associated graded algebras; almost all of these results 
are due to Premet \cite{P1,P2}.

\subsection{Restricted roots}\label{sroots}
For $\al \in (\t^e)^*$, let $\g_\al = \bigoplus_{j \in \Z}
\g_\al(j)$
denote the $\al$-weight space
of $\g$ with respect to $\t^e$. So
$$
\g = \g_0 \oplus \bigoplus_{\al \in \Phi^e} \g_\al
$$
where $\Phi^e \subset (\t^e)^*$ denotes the set of non-zero weights
of $\t^e$ on $\g$. 
Similarly each of the spaces $\p, \m, \n, \h$ and $\k$ decomposes into
$\t^e$-weight spaces.
In the language of \cite[$\S$2]{BG}, $\Phi^e$ is a {\em 
restricted root system}.
It is not a root system in the usual sense;
for example, for $\al \in \Phi^e$
there may be multiples of
$\al$ other than $\pm \al$ that belong to $\Phi^e$.
There is an induced restricted root decomposition
$$
\g^e = \g^e_0 \oplus \bigoplus_{\al \in \Phi^e} \g^e_\al
$$
of the centralizer $\g^e$.
By \cite[Lemma 13]{BG}, $\Phi^e$ is also the set of non-zero weights of
$\t^e$ on $\g^e$, so all the subspaces
$\g^e_\al = \bigoplus_{j \geq 0} \g^e_\al(j)$ in this decomposition are non-zero.
Moreover $\g^e_\al(j)$ is of the same dimension as
$\g^e_{-\al}(j)$; if $j = 0$ this dimension is either $0$ or $1$,
but for $j > 0$ the space $\g^e_\al(j)$ can be bigger.

Recall that $\t^e$ is contained in $\h^e = \g^e(0)$.
We define a {\em dot action} of $\h^e$, hence also of $\t^e$, 
on the vector
space $\widehat\g$
by setting
$$
t \cdot x := [t,x],
\quad
t \cdot y^\ne := [t,y]^\ne,
\quad
t \cdot z^\ch := [t,z]^\ch,
\quad
\langle t \cdot f, z \rangle :=
-\langle f, [t, z] \rangle
$$
for $t \in \h^e, x \in \g, y \in \k, z \in \n$ and $f \in \n^*$.
Using the following lemma, it is routine to check that this 
extends uniquely to an action of $\t^e$ on
$U(\widehat{\g})$ by derivations. 
Moreover the dot action leaves all the subspaces $\widehat\p,
\widetilde\p$ and $\g^e$ invariant, so there are induced dot actions on
$U(\widehat\p)$, $U(\widetilde{\p})$ and $U(\g^e)$ too.

\begin{Lemma}\label{inv}
The adjoint action of $\h^e$ on $\k$ preserves the form
$\lan.|.\ran$.
\end{Lemma}

\begin{proof}
We calculate
\begin{align*}
\lan [t,x] | y \ran + \lan x | [t,y] \ran &= \chi([y,[t,x]] +
[[t,y],x]) =
\chi([t,[y,x]]) \\
&=
(e|[t,[y,x]])
= ([e,t]|[y,x]] = 0,
\end{align*}
for $t
\in \h^e$ and $x, y \in \k$.
\end{proof}

Let $\Z \Phi^e$ denote the $\Z$-submodule of $(\t^e)^*$ generated by $\Phi^e$.
Using the dot action of $\t^e$ we can also decompose all of the
universal enveloping
(super)algebras $U(\widehat{\g})$,
$U(\widehat{\p})$,
$U(\widetilde{\p})$
and $U(\g^e)$ into weight spaces.
For example:
$$
U(\widetilde{\p}) = \bigoplus_{\al \in \Z\Phi^e}
U(\widetilde{\p})_\al
$$
where $U(\widetilde{\p})_\al = \{u \in U(\widetilde{\p})\:|\:
t \cdot u = \al(t) u\text{ for all }t \in \t^e\}$.
Recall the $W$-algebra $U(\g,e)$ is a subalgebra
of $U(\widetilde{\p})$.

\begin{Lemma}\label{Stab}
The dot action of $\h^e$ on $U(\widetilde{\p})$ leaves 
$U(\g,e)$ invariant.
\end{Lemma}

\begin{proof}
We first check that $t
\cdot J \subseteq J$ for $t \in \h^e$. For this we need to
show for any $x \in \n$ that $t \cdot (x - \chi(x) - x^\ne)$ belongs
to $J$. We have that
$
\chi([t,x]) = (e|[t,x]) = ([e,t]|x) = 0
$
as $t$ centralizes $e$. Hence,
$$
t \cdot (x - \chi(x) - x^\ne)
= [t,x] - [t,x]^\ne = [t,x] - \chi([t,x]) - [t,x]^\ne\in J.
$$
Now to prove the lemma, we take $u \in U(\g,e)$, so $[x-x^\ne,u] \in
J$ for all $x \in \n$. We need to show for $t \in \h^e$ that
$[x-x^\ne, t \cdot u] \in J$ for all $x \in \n$ too. Using the
Leibniz rule, we have that
$$
t \cdot [x-x^\ne,u] = [t \cdot (x-x^\ne), u] + [x-x^\ne, t \cdot u]
= [[t,x]-[t,x]^\ne,u] + [x-x^\ne,t \cdot u].
$$
As $[t,x] \in \n$, we know
already that $[[t,x]-[t,x]^\ne,u]\in J$. As $[x-x^\ne,u]\in J$, 
we deduce that $t \cdot
[x-x^\ne,u]\in J$ too. Hence $[x-x^\ne, t \cdot u] \in J$
as required.
\end{proof}

Hence the dot action of $\t^e$ induces a
 {\em restricted root space decomposition}
of $U(\g,e)$:
$$
U(\g,e) = \bigoplus_{\al \in \Z\Phi^e} U(\g,e)_\al.
$$
Actually, we can do rather better: 
the following theorem due to Premet
constructs a canonical embedding of $\t^e$ into $U(\g,e)$
such that the dot action of $\t^e$ on $U(\g,e)$ coincides with 
its adjoint action via this embedding.
This means that $\t^e$ also acts on any $U(\g,e)$-module;
we will use this later on to define weight space decompositions of
 $U(\g,e)$-modules
compatible with the above restricted root space decomposition.
To formulate the theorem following \cite[$\S$2.5]{P2},
let $z_1,\dots,z_{2s}$ be a symplectic basis for $\k$,
so that $\lan z_i | z_j^* \ran = \delta_{i,j}$ for all $1\leq
i,j\leq 2s$ where
$$
z_j^* := \left\{
\begin{array}{rl}
z_{j+s}&\text{for $j=1,\dots,s$,}\\
-z_{j-s}&\text{for $j=s+1,\dots,2s$.}
\end{array}\right.
$$

\begin{Theorem}\label{lif}
There is a $\t^e$-equivariant linear map $\theta: \g^e \hookrightarrow
U(\widetilde{\p})$ such that
$$
\theta(x) =
\left\{
\begin{array}{ll}
x + {\textstyle \frac{1}{2}} \sum_{i=1}^{2s} [x, z_i^*]^\ne z_i^\ne
&\text{if $x \in \h^e$,}\\
x&\text{otherwise,}
\end{array}\right.
$$
for each homogeneous element $x \in \g^e$. The map $\theta$ does not
depend on the choice of basis $z_1,\dots,z_{2s}$. Moreover:
\begin{itemize}
\item[(1)]
$[\theta(x), u] = x \cdot u$ for any $x \in \h^e$ and $u \in
U(\widetilde{\p})$;
\item[(2)] $[\theta(x), \theta(y)] = \theta([x,y])$
for each $x, y \in \g^e$;
\item[(3)] $\theta(x)$ belongs to $U(\g,e)$ 
for every $x \in \h^e$.
\end{itemize}
\end{Theorem}

\begin{proof}
This is proved in \cite[$\S$2.5]{P2} but we repeat the details
since our setup is slightly different.
It is routine
that 
$\theta$ is independent of the
choice of the symplectic basis.
Moreover, as the symplectic form on
$\k$ is $\t^e$-invariant, 
we may choose the
symplectic basis to be a $\t^e$-weight basis with the weight of
$z_i$ being the negative of the weight of $z_i^*$ for each $i$. It
is then clear that $\theta$ is $\t^e$-equivariant.

We next check
(1) for a fixed $x \in \h^e$. Both the
maps $u \mapsto [\theta(x), u]$ and $u \mapsto x \cdot u$ are
derivations, so it suffices to check that $[\theta(x), y] = x \cdot
y$ for each $y \in \p$ and that $[\theta(x), (z_j^*)^\ne] = x \cdot
(z_j^*)^\ne$ for each $j=1,\dots,2s$. The first of these statements
is obvious, since both $[\theta(x), y]$ and $x \cdot y$ equal
$[x,y]$. For the second, we calculate using Lemma~\ref{inv}:
\begin{align*}
[\theta(x), (z_j^*)^\ne] &= \textstyle \frac{1}{2} \sum_i \lan
[x,z_i^*] | z_j^* \ran z_i^\ne +
\frac{1}{2} [x,z_j^*]^\ne\\
&= \textstyle \frac{1}{2} \sum_i \lan [x,z_j^*] | z_i^* \ran z_i^\ne
+ \frac{1}{2} [x,z_j^*]^\ne = [x,z_j^*]^\ne.
\end{align*}

For (2), take homogeneous $x, y \in \g^e$. If
$x$ or $y$
belongs to $\g^e(j)$ for some $j > 0$ then (2) is obvious 
so assume $x,y\in \h^e$. Then using (1) we get 
$$
[\theta(x), \theta(y)] = x \cdot \theta(y)
=
[x,y]
 + {\textstyle \frac{1}{2}} {\textstyle\sum_i} [x,[y, z_i^*]]^\ne z_i^\ne
 + {\textstyle \frac{1}{2}} {\textstyle\sum_i} [y, z_i^*]^\ne [x,z_i]^\ne.
$$
Now we simplify the last term:
\begin{align*}
\textstyle \sum_i [y,z_i^*]^\ne [x,z_i]^\ne
&=
-\textstyle \sum_i [y,z_i]^\ne [x,z_i^*]^\ne\\
&= -\textstyle \sum_{i,j} \lan [x,z_i^*] | z_j^* \ran
[y,z_i]^\ne z_j^\ne\\
&= -\textstyle \sum_{i,j} \lan [x,z_j^*] | z_i^* \ran
[y,z_i]^\ne z_j^\ne\\
&=
-\textstyle \sum_{i}
[y,[x,z_i^*]]^\ne z_i^\ne\\
&=
\textstyle\sum_{i}
[[x,y],z_i^*]^\ne z_i^\ne
- \sum_i [x,[y,z_i^*]]^\ne z_i^\ne.
\end{align*}
Thus 
$$
[\theta(x),\theta(y)]  =
[x,y] + \textstyle\frac{1}{2} \textstyle\sum_i [[x,y],z_i^*]^\ne z_i^\ne
=
\theta([x,y]).
$$

To check (3), take $x \in \h^e$.
We need to show that $[y-y^\ne,\theta(x)] \in J$ for all $y \in \n$. If $y \in 
\m$, 
then $[y-y^\ne,\theta(x)] = [y,x] \in J$, as $\chi([y,x]) = 0$ using the fact that $x \in \g^e$. If instead $y \in \k$, then using
(1) we have that
\begin{align*}
[y-y^\ne, \theta(x)] &=
[y,x] + [\theta(x), y^\ne]= [y,x] + x \cdot y^\ne \\
&= [y,x] + [x,y]^\ne = [y,x] - \chi([y,x]) - [y,x]^\ne
\end{align*}
which does belong to $J$.
\end{proof}

\begin{Remark}\label{thetam}
One can modify the map $\theta$ in
Theorem~\ref{lif} by composing 
with an automorphism
of $U(\widetilde\p)$ of the form
$x \mapsto x + \eta(x), y^\ne \mapsto y^\ne$ 
for some character $\eta$ of $\p$.
Providing the character $\delta$
in $\S$\ref{sCartan} below is replaced 
by $\delta-\eta$,
all subsequent results remain true exactly as formulated for $\theta$ modified in this way.
This extra degree of freedom will be useful when we discuss type A
in $\S$\ref{sA} below.
\end{Remark}

\subsection{The Kazhdan filtration}\label{sKazhdan}
Now we introduce the first important filtration,
working to start with
in terms of the Whittaker model realization 
$Q^{\n}$ of the $W$-algebra
following \cite{GG}. The {\em Kazhdan filtration}
$$
\cdots \subseteq \F_i U(\g) \subseteq \F_{i+1}U(\g) \subseteq\cdots
$$
on $U(\g)$ is defined by declaring that $x \in \g(j)$ is of
{\em Kazhdan degree} $(j+2)$.
So $\F_i U(\g)$ is the
span of the monomials $x_1 \cdots x_n$ for 
$n \geq 0$ and 
$x_1 \in \g(j_1)$, $\dots$, $x_n \in \g(j_n)$
such that $(j_1+2) + \cdots + (j_n+2) \leq i$. The associated graded
algebra $\gr U(\g)$ is the symmetric algebra $S(\g)$ viewed as a
graded algebra via the {\em Kazhdan grading} on $\g$ in which $x \in
\g(j)$ is of degree $(j+2)$.

The Kazhdan filtration on $U(\g)$ induces a filtration on the left
ideal $I$ and on the quotient $Q = U(\g) / I$ such that $\gr Q =
S(\g) / \gr I$. 
We use the bilinear form $(.|.)$ to 
identify $S(\g)$ with the algebra $\C[\g]$ of regular
functions
on the affine variety $\g$. Then $\gr I$ is the ideal generated by
the functions 
$\{x - \chi(x)\:|\:x \in \m\}$, i.e. the 
ideal of all functions in $\C[\g]$ vanishing on the closed
subvariety $e + \m^\perp$ of $\g$. In this way, we have identified
$$
\gr Q = \C[e + \m^\perp].
$$

Let $N$ be the
closed subgroup of $G$
corresponding to the subalgebra
$\n$. The adjoint action of $N$ on $\g$ leaves $e + \m^\perp$
invariant, so we get induced an action of $N$ on
$\C[e + \m^\perp]$ by automorphisms; the resulting action of
$\n$ on $\C[e + \m^\perp]$ by derivations coincides under the above
identification with the action of $\n$ on $\gr Q$ induced by the
adjoint action of $\n$ on $Q$ itself.
By \cite[Lemma 2.1]{GG}, the action of
$N$ on $e + \m^\perp$ is regular and
the {\em Slodowy
slice} $$ e + \g^f \subseteq e + \m^\perp
$$
gives a set of representatives for the orbits of $N$ on
$e+\m^\perp$. 
So restriction of
functions defines an isomorphism from the invariant
subalgebra $\C[e+\m^\perp]^N$ onto
$\C[e+\g^f]$.  
Finally, the Kazhdan filtration on
$Q$ induces an algebra filtration on $Q^{\n}$ so that 
$\gr (Q^{\n})$ is identified with a graded
subalgebra of $\gr Q = \C[e+\m^\perp]$. Now 
the {\em PBW theorem}\, for $Q^{\n}$ \cite[Theorem 4.1]{GG} gives:
$$
\gr (Q^{\n}) = \C[e+\m^\perp]^N \cong \C[e+\g^f].
$$

The next goal is to
reformulate this PBW theorem for $U(\g,e)$ directly.
To do
this, extend the Kazhdan filtration on $U(\g)$ to
$U(\widetilde{\g})$ by declaring that elements of $\k^\ne$ are of
degree $1$. Setting $\F_i U(\widetilde{\p}) := U(\widetilde{\p})
\cap \F_i U(\widetilde{\g})$, 
we get an induced Kazhdan filtration on the subalgebra
$U(\widetilde{\p})$. We can obviously identify $\gr
U(\widetilde{\g})$ 
and $\gr U(\widetilde{\p})$
with $S(\widetilde{\g})$ and $S(\widetilde{\p})$, 
both graded via the
analogous {\em Kazhdan grading} in which $x \in \g(j)$ 
is of
Kazhdan degree $j+2$ and $y^\ne \in \k^\ne$ is of Kazhdan degree $1$. 
The Kazhdan grading on 
$\widetilde{\p}$ only involves positive degrees,
so the Kazhdan filtration on $U(\widetilde{\p})$ is
{strictly positive} in the sense that 
$\F_0 U(\widetilde{\p}) = \C$ and
$\F_i U(\widetilde{\p}) = 0$ for $i < 0$.
We get an induced strictly positive filtration
$$
\F_0 U(\g,e) \subseteq \F_1 U(\g,e) \subseteq \cdots
$$
on $U(\g,e)$ such that
$\gr U(\g,e)$ is canonically identified with a graded subalgebra of
$\gr U(\widetilde{\p}) = S(\widetilde{\p})$.
The point now is that there is a
commutative diagram:
$$
\xymatrix{& \gr U(\widetilde{\g}) =
 S(\widetilde{\g}) \ar@{>>}[d]_{\pr} \ar[dr] \\
\gr U(\g,e) =  S(\widetilde{\p})^\n \ar@{^{(}->}[r] \ar[dr]_-{\sim} &
\gr U(\widetilde{\p}) = S(\widetilde{\p}) \ar[r]^-{\sim}
\ar@{>>}[d]_{\zeta} &
\C[e+\m^\perp] \ar@{>>}[d]^{\res} & \C[e+\m^\perp]^N \ar@{_{(}->}[l] \ar[dl]^-{\sim} \\
 & S(\g^e) \ar[r]^-{\sim} & \C[e+\g^f] \\ }
$$

There are several things still to be explained in this diagram.
The algebra homomorphisms $S(\widetilde{\g}) \twoheadrightarrow \C[e+\m^\perp]$ and
$S(\widetilde{\p}) \stackrel{\sim}{\rightarrow} \C[e+\m^\perp]$ are
induced by the multiplication maps
$U(\widetilde{\g}) \twoheadrightarrow Q, u \mapsto u (1+I)$ and
$U(\widetilde{\p}) \stackrel{\sim}{\rightarrow} Q, u \mapsto u(1+I)$, 
on passing to the associated graded objects.
Explicitly, they send $x \in \g$ or
$\p$ to the function $z \mapsto (x|z)$ and $y^\ne \in \k^\ne$ to the
function $z \mapsto (y|z)$.

The homomorphism $\pr:S(\widetilde{\g}) \twoheadrightarrow S(\widetilde{\p})$
is
induced by $\Pr:U(\widetilde{\g})
\twoheadrightarrow U(\widetilde{\p})$ on passing to the associated
graded algebras. Explicitly, $\pr$ is the identity on elements of
$\widetilde{\p}$ and maps $x \in \n$ to $\chi(x) + x^\ne$. The top
triangle in the diagram commutes even 
before passing to the associated graded objects, as we observed that
$\Pr(u) (1+I) = u (1+I)$ for $u \in U(\widetilde{\g})$
in the proof of Theorem~\ref{t1}.

The twisted adjoint action of $\n$ on $U(\widetilde{\g})$ induces a
graded action of $\n$ on $S(\widetilde{\g})$ by derivations, such
that
 $x \in \n$ maps
$y \in \widetilde{\g}$ to $[x-x^\ne,y]$. This action factors through
the map $\pr$ to induce an action of $\n$ on $S(\widetilde{\p})$ by
derivations, such that
 $x \in \n$ maps
$y \in \widetilde{\p}$ to $\pr([x-x^\ne,y])$. We let
$S(\widetilde{\p})^{\n}$ denote the invariant subalgebra for this
action.

By Lemma~\ref{nat}, the map
$S(\widetilde{\g}) \twoheadrightarrow \C[e+\m^\perp]$ intertwines
the twisted adjoint action of $\n$ on $S(\widetilde{\g})$ with the
action of $\n$ on $\C[e+\m^\perp]$ derived from the action of the
group $N$. Hence the map $S(\widetilde{\p})
\stackrel{\sim}{\rightarrow} \C[e+\m^\perp]$ restricts to an
isomorphism $S(\widetilde{\p})^{\n}
\stackrel{\sim}{\rightarrow}\C[e+\m^\perp]^N$.
By Theorem~\ref{t1} and the PBW theorem for $Q^{\n}$, the
isomorphism $S(\widetilde{\p}) \stackrel{\sim}{\rightarrow}
\C[e+\m^\perp]$ maps $\gr U(\g,e)$ isomorphically onto
$\gr (Q^{\n}) = \C[e+\m^\perp]^N$. This shows that $\gr
U(\g,e) = S(\widetilde{\p})^{\n}$.

Let $\zeta:S(\widetilde{\p}) \twoheadrightarrow S(\g^e)$ be the
homomorphism induced by the projection $\widetilde{\p}
\twoheadrightarrow \g^e$ along the
decomposition from Lemma~\ref{wpd}.
We have that
$$
\qquad([f,x]|z) = (x|[z,f]) = (x|[e,f]) = (x|h) = 0
$$
for all $x \in \bigoplus_{j \geq 1} \g(j)$ and $z \in e + \g^f$.
So the image of any element of
$\ker \zeta$ under the isomorphism $S(\widetilde{\p})
\stackrel{\sim}{\rightarrow} \C[e+\m^\perp]$ annihilates 
$e + \g^f$. Hence there is an induced homomorphism
$S(\g^e) \rightarrow \C[e+\g^f]$ making the bottom square in the
diagram commute. This
homomorphism maps $x \in \g^e$
to the function $z \mapsto (x|z)$, from which it is easy to check that it is
actually an isomorphism.

Finally, we have already noted that the map $\res$ sends
$\C[e+\m^\perp]^N$ isomorphically onto the coordinate algebra
$\C[e+\g^f]$ of the Slodowy slice.
So the restriction of
$\zeta$ gives an isomorphism
$S(\widetilde{\p})^{\n}\stackrel{\sim}{\rightarrow}S(\g^e)$.
This completes the justification of the above diagram,
and we have now derived the following convenient
reformulation of \cite[Theorem 4.1]{GG}:

\begin{Lemma}\label{pbwnew}
We have that $\gr U(\g,e) = S(\widetilde\p)^\n$.
Moreover, the restriction of $\zeta$
is an isomorphism 
of graded algebras
$$
\zeta: S(\widetilde\p)^{\n}
\stackrel{\sim}{\rightarrow} S(\g^e),
$$ 
where the grading on $S(\g^e)$ is induced by the Kazhdan grading on $\g^e$. 
\end{Lemma}

The dot action of $\t^e$ on $\widetilde\p$
extends to an action
on $S(\widetilde\p)$ by derivations, which coincides with 
the action induced by the 
dot action of $\t^e$ on $U(\widetilde\p)$ on passing to the associated graded
algebra. Hence the $\t^e$-weight space decomposition
$$
S(\widetilde{\p}) = \bigoplus_{\al \in
\Z\Phi^e} S(\widetilde{\p})_\al
$$
satisfies
$\gr (U(\widetilde{\p})_\al) = S(\widetilde{\p})_\al$. In
view of Lemma~\ref{Stab}, $U(\g,e)$ 
is a $\t^e$-submodule of $U(\widetilde{\p})$, hence $\gr U(\g,e)= 
S(\widetilde{\p})^{\n}$ is a $\t^e$-submodule of
$S(\widetilde{\p})$
and the induced decomposition
$$
\gr U(\g,e) = \bigoplus_{\al
\in \Z\Phi^e} S(\widetilde{\p})^{\n}_\al
$$
satisfies $\gr (U(\g,e)_\al)=
S(\widetilde{\p})^{\n}_\al$.
Now we can deduce the following {\em PBW theorem} for $U(\g,e)$,
which is essentially
\cite[Theorem 4.6]{P1}. We remark that the original proof in \cite{P1} involved lifting from characteristic $p$,
whereas we are deducing it ultimately from \cite[Theorem 4.1]{GG}.

\begin{Theorem}\label{Th}
There exists a (non-unique) $\t^e$-equivariant linear map
$$
\Theta:\g^e \hookrightarrow U(\g,e)
$$
such that 
$\Theta(x) \in\F_{j+2} U(\g,e)$ and
$\zeta(\gr_{j+2}
\Theta(x)) = x$ for each $x \in \g^e(j)$, and
$\Theta(t) = \theta(t)$ 
for each $t \in \h^e$.
Moreover, 
if $x_1,\dots,x_t$ is a homogeneous basis of
$\g^e$ with $x_i \in \g^e(n_i)$ then
the monomials 
$$
\{\Theta(x_{i_1})\dots\Theta(x_{i_k})\:|\:k \geq 0, 
1\leq i_1\leq\dots\leq i_k\leq t,
n_{i_1}+\cdots+n_{i_k}+2k \leq j
\}
$$
form a basis for $\F_j U(\g,e)\:(j \geq 0)$.
\end{Theorem}

\begin{proof}
Let $x_1,\dots,x_t$ be a
basis for $\g^e$ with
$x_i \in \g(n_i)$ of $\t^e$-weight $\gamma_i$. 
As the decomposition from Lemma~\ref{wpd}
is a direct sum of $\t^e$-modules, the isomorphism
$\zeta:S(\widetilde{\p})^\n \stackrel{\sim}{\rightarrow} S(\g^e)$ 
from Lemma~\ref{pbwnew}
is a $\t^e$-equivariant isomorphism of Kazhdan graded vector 
spaces. 
For each $i$, 
let $\hat{x}_i \in S(\widetilde{\p})_{\gamma_i}^{\n}$ be the
unique element with $\zeta(\hat{x}_i) = x_i$. 
So 
$\hat{x}_i$ belongs to $\gr_{n_i+2}
(U(\g,e)_{\gamma_i})$. Hence there is
a lift $\Theta(x_i) \in \F_{n_i+2} U(\g,e)_{\gamma_i}$ with
$\gr_{n_i+2} \Theta(x_i) = \hat{x}_i$, i.e.\ $\zeta(\gr_{n_i+2}
\Theta(x_i)) = x_i$. 
Moreover if $n_i = 0$ then
by Lemma~\ref{lif}(3) we know that
$\theta(x_i) \in \F_2 U(\g,e)_{\gamma_i}$
and $\zeta(\gr_2 \theta(x_i)) = x_i$, so
we can choose $\Theta(x_i)$ to be $\theta(x_i)$.
Now extend by linearity to obtain the desired map $\Theta$.
The final statement then follows immediately from Lemma~\ref{pbwnew}.
\end{proof}

\subsection{The good filtration}\label{sGood}
Now we turn our attention to the second
filtration on $U(\g,e)$, which we refer to as the {\em good
filtration}. To define it, the good grading on $\p$ induces a
grading on $U(\p)$. We extend this to the {\em good grading}
$$
U(\widetilde{\p}) = \bigoplus_{j \geq 0} U(\widetilde{\p})(j)
$$
on $U(\widetilde{\p})$
by
declaring that elements of $\k^\ne$ are of degree $0$. The
subalgebra $U(\g,e)$ is {\em not} a graded
subalgebra in general, but the good grading at least induces
the {\em good filtration} 
$$
\F'_0 U(\g,e) \subseteq \F'_1 U(\g,e) \subseteq \cdots
$$
on $U(\g,e)$,
where $\F'_j U(\g,e) = U(\g,e) \cap \bigoplus_{i \leq j} U(\widetilde{\p})(i)$. 
The associated graded algebra 
$\gr' U(\g,e)$ is then canonically identified with a graded
subalgebra of $U(\widetilde{\p})$.
Our goal is to identify this associated graded algebra with 
$U(\g^e)$, a result which is a slight variation on 
\cite[Proposition 2.1]{P2}.

\begin{Lemma}\label{beep}
Let $\theta$ and $\Theta$ be as in Theorems~\ref{lif} and
\ref{Th}.
For $x \in \g^e(j)$ we have that
$\Theta(x) \in \F'_j U(\g,e)$ and
$\gr'_j \Theta(x) = \theta(x)$.
\end{Lemma}

\begin{proof}
Take $0 \neq x \in \g^e(j)$ for $j \geq 0$.
Let $\hat x \in 
S(\widetilde{\p})^{\n}$ be the unique element 
such that $\zeta(\hat x) = x$.
We claim to start with that
$$
\hat{x} \equiv x \pmod{\textstyle\bigoplus_{n \geq 2}
S^n(\widetilde{\p})}.
$$
To see this, note $\zeta$ is a graded map with
respect to the Kazhdan grading and $x$ is in Kazhdan degree $j+ 2$,
so we certainly have
that $\hat x \equiv y 
\pmod{\bigoplus_{n \geq 2}
S^n(\widetilde{\p})}$ for some unique $y \in \g(j)$
with $\zeta(y) = x$. Now we just need to show that $y$ centralizes
$e$, i.e. $\zeta(y) = y$. If not then $[e,y] \neq 0$, 
so we can find some $z \in \g(-j-2)$ 
such that $([e,y]|z) \neq 0$. 
Then
$$
\pr([y,z]) = \chi([y,z]) = (e|[y,z]) = ([e,y]|z) \neq 0.
$$
On the other hand if $x_1 \cdots x_n$ is any monomial in
$S^n(\widetilde{\p})$ of Kazhdan degree $j+2$ for some $n \geq 2$, we
have that
$$
\pr([x_1 \cdots x_n,z]) =
\sum_{i=1}^n x_1 \cdots x_{i-1} \pr([x_i,z]) x_{i+1} \cdots x_n.
$$
As $n \geq 2$ we know that $x_i$ is of Kazhdan degree strictly
smaller than $j+2$, so $[x_i,z]$ lies in $\g(k)$ for some $k < -2$.
Hence $\pr([x_i,z]) = 0$ so $\pr([x_1\cdots x_n,z]) = 0$ too. 
So $\pr([\hat{x},z]) =
\pr([y,z]) \neq 0$.
But $\hat{x}$ lies in
$S(\widetilde{\p})^{\n}$, so by the definition of the
$\n$-action we have that $\pr([z,\hat{x}]) =  0$,
giving the desired contradiction.

Now to prove the lemma, take $x \in \g^e(j)$.
If $j = 0$ then $\Theta(x) = \theta(x)$ by Theorem~\ref{Th},
and the conclusion is clear.
So assume that $j > 0$.
We need to show that $\Theta(x) = \theta(x) + (\dagger)$
where $(\dagger)$ belongs to $\bigoplus_{i
< j} U(\widetilde{\p})(i)$. 
By the claim,
the lift $\hat{x}$ is equal to $x$ plus a linear combination
of monomials of the form $y_1 \cdots y_rz_1^\ne \cdots z_s^\ne$ for
$r+s \geq 2$ and elements $y_i \in \p(k_i)$ and $z_i \in
\k$ with 
$(k_1+2) + \cdots + (k_r + 2) + s = j+2$.
Hence the lift $\Theta(x)$ of $\hat x$ to $U(\g,e)$ is equal to $x$
plus a linear combination of such monomials $y_1 \cdots y_r z_1^\ne
\cdots z_s^\ne$ plus some element $u \in \F_{j+1}
U(\widetilde{\p})$. Each $y_1 \cdots y_r z_1^\ne \cdots z_s^\ne$ is
of degree $k_1+\cdots + k_r = j - 2r -s+2$. As $r+s \geq 2$,
we deduce that $k_1+\cdots+k_r$ 
is either zero or it is 
strictly less than $j$, so all the 
$y_1 \cdots y_r z_1^\ne \cdots z_s^\ne$ terms belong to $\bigoplus_{i <
j} U(\widetilde{\p})(i)$. 
Finally the element $u$ is itself 
a linear combination of monomials
$y_1\cdots y_r z_1^\ne \cdots z_s^\ne$ 
with $y_i \in {\p}(k_i)$ and $z_i \in \k$
such that $(k_1+2)+\cdots+(k_r+2) + s \leq j+1$.
Hence $k_1+\cdots+k_r \leq j-2r-s+1$, 
so again we have that  
$k_1+\cdots+k_r$ is either zero or strictly less than $j$.
This shows that $u$ belongs to $\bigoplus_{i <
j} U(\widetilde{\p})(i)$ too.
\end{proof}

\begin{Theorem}\label{TGoodFil}
The  homomorphism $U(\g^e) \rightarrow
U(\widetilde{\p})$ 
induced by the Lie algebra homomorphism
$\theta$ from
Theorem~\ref{lif} 
defines a $\t^e$-equivariant graded algebra
isomorphism 
$$
\theta:U(\g^e)\stackrel{\sim}{\rightarrow}
\gr' U(\g,e),
$$
viewing $U(\g^e)$ as a graded algebra via the good 
grading.
\end{Theorem}

\begin{proof}
Pick $\Theta$ and a homogeneous basis $x_1,\dots,x_t$ for $\g^e$
as in Theorem~\ref{Th}, so that 
the monomials $\Theta(x_{i_1}) \cdots \Theta (x_{i_k})$ for
all $1 \leq i_1 \leq \cdots \leq i_k \leq t$ form a basis for
$U(\g,e)$. Lemma~\ref{beep}
 implies that
$$
\Theta(x_{i_1}) \cdots \Theta(x_{i_k}) =
\theta(x_{i_1} \cdots x_{i_k}) + (\dagger)
$$
where the first term on the right hand side lies in
$U(\widetilde{\p})(n_{i_1}+\cdots+n_{i_k})$ and the term $(\dagger)$ lies in
the sum of all strictly lower graded components. Hence
the monomials $\{\theta(x_{i_1} \cdots x_{i_k})\:|\:
1 \leq i_1 \leq \cdots \leq i_k \leq t\}$ give a homogeneous basis for
$\gr' U(\g,e)$.
By the PBW theorem for $U(\g^e)$, the same monomials
give a homogeneous basis for $\theta(U(\g^e))$.
\end{proof}

\section{Highest weight theory}
This section contains the main new results of the paper. We are going to
define Verma modules
and explain their relevance to the problem of classifying finite
dimensional irreducible $U(\g,e)$-modules.

\subsection{``Cartan subalgebra''}\label{sCartan}
Recall from $\S$\ref{sroots}
that the restricted root system $\Phi^e$ is the set of non-zero weights
of $\t^e$ on $\g = \g_0 \oplus \bigoplus_{\al \in \Phi^e} \g_\al$.
The zero weight space $\g_0$ is the centralizer of the toral subalgebra
$\t^e$ in $\g$, so it
is a Levi factor of a parabolic 
subalgebra of $\g$. According to Bala--Carter theory 
\cite[5.9.3, 5.9.4]{C}, $e$ is a
distinguished nilpotent element of $\g_0$, i.e. the only
semisimple elements of $\g_0$ that centralize $e$
belong to the center of $\g_0$.
It is also clear that $h$ and $f$ lie in
$\g_0$.

By \cite[Lemma 19]{BG}, our fixed good grading $\g = \bigoplus_{j \in \Z} \g(j)$
coincides with the eigenspace decomposition
of $\ad(h+p)$ for some element $p \in \t^e$.
As $p$ centralizes $\g_0$ it follows that
the induced grading 
$\g_0 = \bigoplus_{j \in \Z} \g_0(j)$
coincides simply with the $\ad h$-eigenspace decomposition of $\g_0$.
Of course this is another good grading for $e$, viewed now as
an element of the smaller reductive Lie algebra $\g_0$.
Moreover by \cite[5.7.6]{C} it is
an {even} grading, which means that 
$\widetilde{\g}_0 = \g_0$, $\widetilde{\p}_0 = \p_0$,
$\m_0 =\n_0$ and $\k_0 = \{0\}$.
In particular, the finite $W$-algebra
$U(\g_0,e)$ associated to $e\in\g_0$ is defined simply by
$$
U(\g_0,e):=\{u\in U(\p_0)\mid  \Pr_0([x,u]) = 0 \text{ for all }x \in \m_0\},
$$
where $I_0$ is the left ideal of $U(\g_0)$ generated
by the elements $\{x-\chi(x)\:|\:x\in \m_0\}$,
and 
$\Pr_0$ is the projection along the decomposition $U(\g_0)=
U(\p_0)\oplus I_0$.
This finite $W$-algebra is going to play the role of Cartan subalgebra
in our highest weight theory.
However, unlike in the case $e = 0$ it does not
embed obviously as a subalgebra of $U(\g,e)$; instead we will
realize it as a certain section.

Before we can do this, we need to fix one more critically
important choice:
let $\q$ be a parabolic subalgebra of $\g$ with Levi factor $\g_0$.
We stress that there is
often more than one conjugacy class of choices for $\q$, 
unlike in the case $e=0$
when there is just one conjugacy class of Borel subalgebras containing $\t$.
In the language of \cite[$\S$2]{BG}, the choice of $\q$ determines
a system $\Phi^e_+$ of positive roots
in the restricted root system $\Phi^e$, namely,
$\Phi^e_+ := \{\alpha \in \Phi^e\:|\:\g_\alpha \subseteq \q\}$.
Setting $\Phi^e_- := - \Phi^e_+$,
we define
$\g_{\pm} := \bigoplus_{\alpha \in \Phi_{\pm}^e} \g_\alpha$,
so that 
$$
\g = \g_- \oplus \g_0 \oplus \g_+,\qquad
\q = \g_0 \oplus \g_+.
$$
The choice $\Phi^e_+$ of positive roots
induces a {\em dominance ordering} $\leq$ on $(\t^e)^*$:
$\mu\leq\la$ if
$\la-\mu\in \Z_{\geq 0}\Phi^e_+$.

In this paragraph, we 
let $\a$ denote one of $\widehat{\g}, \widehat{\p},
\widetilde{\p}$ or $\g^e$.
Recall from $\S$\ref{sroots} that the dot actions of 
$\t^e$ on $\a$ and its universal enveloping (super)algebra $U(\a)$
induce decompositions
$\a = \a_0 \oplus \bigoplus_{\al \in \Phi^e} \a_\al$
and $U(\a) = \bigoplus_{\al \in \Z\Phi^e} U(\a)_\al$.
In particular, $U(\a)_0$,
the zero weight space of $U(\a)$ with 
respect to the dot action, is a subalgebra of $U(\a)$.
Also let $U(\a)_\sharp$ (resp.\ $U(\a)_\flat$) denote the left (resp.\ right)
ideal of $U(\a)$ generated by the root spaces $\a_\al$ for
$\al \in \Phi^e_+$ (resp.\ $\al \in \Phi^e_-$).
Let
$$
U(\a)_{0, \sharp} := U(\a)_0 \cap U(\a)_\sharp,
\qquad
U(\a)_{\flat,0} := U(\a)_\flat \cap U(\a)_0,
$$
which are obviously left and right ideals of $U(\a)_0$, respectively.
By the PBW theorem for non-linear Lie algebras, we actually have that
$$
U(\a)_{0,\sharp} = U(\a)_{\flat,0},
$$
hence $U(\a)_{0,\sharp}$ is a two-sided ideal of $U(\a)_0$.
Moreover $\a_0$ 
is a subalgebra of $\a$,
and by the PBW theorem again we have that
$U(\a)_0 = U(\a_0) \oplus U(\a)_{0,\sharp}$.
The projection along this decomposition
defines a surjective algebra homomorphism
$$
\pi:U(\a)_0 \twoheadrightarrow U(\a_0)
$$ 
with $\ker\pi = U(\a)_{0,\sharp}$.
Hence $U(\a)_0 / U(\a)_{0,\sharp} \cong U(\a_0)$.

We can repeat some but not all of the 
preceding discussion for the $W$-algebra $U(\g,e)$ itself.
To make things as explicit as possible, let us 
choose a homogeneous $\t^e$-weight basis
$f_1,\dots,f_m,h_1,\dots,h_l,e_1,\dots,e_m$ of $\g^e$ so that the
weight of $f_i$ is $-\ga_i\in\Phi^e_-$, the weight of $e_i$ is
$\ga_i\in\Phi^e_+$, and $h_1,\dots,h_l\in\g^e_0$;
the weights $\ga_i$ here are not necessarily distinct. Choosing
an embedding $\Theta$ as in Theorem~\ref{Th},
we get the corresponding elements
$F_i:=\Theta(f_i)\in U(\g,e)_{-\ga_i}$, $E_i:=\Theta(e_i)\in
U(\g,e)_{\ga_i}$, and $H_j:=\Theta(h_j)\in U(\g,e)_0$. 
For $\aa\in\Z_{\geq 0}^m$, we write $F^{\aa}$ for $F_1^{a_1}\dots
F_m^{a_m}$, and define $H^\bb$ and $E^\cc$ for $\bb\in\Z_{\geq 0}^l$
and $\cc\in\Z_{\geq 0}^m$ similarly. We get the following 
PBW basis for $U(\g,e)$:
\begin{equation*}
\{F^\aa H^\bb E^\cc\:|\: \aa,\cc\in\Z_{\geq 0}^m,\bb\in\Z_{\geq 0}^l\}.
\end{equation*}
The subspace
$U(\g,e)_\al$ in the restricted root space decomposition
has basis given by all the
PBW monomials $F^\aa H^\bb E^\cc$ such that
$\sum_{i} (c_i-a_i)\ga_i = \al$.

Now we define $U(\g,e)_\sharp$ (resp.\ $U(\g,e)_\flat$) to be the
left (resp.\ right) ideal of $U(\g,e)$ generated by
$E_1,\dots,E_m$ (resp.\ $F_1,\dots,F_m$).
Note that $U(\g,e)_\sharp$ (resp.\ $U(\g,e)_\flat$)
is equivalently the left (resp.\ right) ideal of
$U(\g,e)$ generated by all $U(\g,e)_\al$ for $\al \in \Phi^e_+$
(resp.\ $\al \in \Phi^e_-$), so it does not depend on the
explicit choice of the basis.
Set
$$
U(\g,e)_{0,\sharp} := U(\g,e)_0 \cap U(\g,e)_\sharp,
\qquad
U(\g,e)_{\flat,0} := U(\g,e)_\flat \cap U(\g,e)_0,
$$
which are obviously left and right ideals of the zero weight space 
$U(\g,e)_0$, respectively.
The PBW monomials $F^\aa H^\bb E^\cc$ with
$\cc\neq{\mathbf 0}$ (resp.\ $\aa\neq{\mathbf 0}$) form a basis of 
$U(\g,e)_\sharp$ (resp.\ $U(\g,e)_\flat$), and the
PBW monomials $F^\aa H^\bb E^\cc$ with $\sum_i(c_i-a_i)\ga_i = 0$ form a basis of $U(\g,e)_0$. It follows that 
$$
U(\g,e)_{0,\sharp} = U(\g,e)_{\flat,0},
$$
hence it is a two-sided ideal of  $U(\g,e)_0$.
Moreover the cosets of  the PBW monomials of the
form $H^\bb$ form a basis of the quotient algebra
$U(\g,e)_0/U(\g,e)_{0,\sharp}$. 
However the PBW monomials $H^\bb$ need not span a {\em subalgebra} of 
$U(\g,e)$, unlike the situation for the algebras $U(\a)$ 
discussed earlier.

The goal now is to prove using the BRST cohomology definition
of the $W$-algebra that
$U(\g,e)_0 / U(\g,e)_{0, \sharp}$ is canonically isomorphic to 
$U(\g_0,e)$.
The isomorphism involves a shift in the spirit of
Corollary~\ref{rhs}: let
$$\gamma := \sum_{\substack{1 \leq i \leq r \\
\be_i|_{\t^e}\in \Phi^e_-}}\be_i,
\qquad\qquad
\delta := \sum_{\substack{1 \leq i \leq r \\
\be_i|_{\t^e}\in \Phi^e_-\\ d_i \geq 2}}\be_i
+
{\textstyle\frac{1}{2}}\sum_{\substack{1 \leq i \leq r \\ \be_i|_{\t^e} \in \Phi^e_-
\\
d_i =1}} \be_i,
$$
recalling from $\S$2 that $b_1,\dots,b_r$ is a homogeneous basis
for $\n$ with $b_i \in \g(-d_i)$ of
weight $\be_i \in \t^*$.

\begin{Lemma}\label{es}
$\gamma$ and $\delta$ extend uniquely to 
characters of $\p_0$.
\end{Lemma}

\begin{proof}
For $x \in \h_0$, $\ad x$ leaves the subspace
$\n_- = \bigoplus_{\al \in \Phi^e_-} \n_\al$ invariant, so the map
$$
\ga:\h_0 \rightarrow \C, \qquad x \mapsto \tr(\ad x|_{\n_-})
$$
is a well-defined Lie algebra homomorphism
with $\gamma(x) := \tr(\ad x|_{\n_-})$.
Extending by zero on the
nilradical of $\p_0$, this defines the required
homomorphism $\gamma:\p_0 \rightarrow \C$.
The construction of $\delta$ is similar.
\end{proof}

\begin{Lemma}
\label{bed}
The following diagram commutes:
$$
\begin{CD}
U(\g^e_0)@>\theta>> U(\widetilde\p)_0\\
@VS_{-\delta}VV@VV\pi V\\
U(\p_0)@<<S_{-\gamma}< U(\p_0).
\end{CD}
$$
\end{Lemma}
\begin{proof}
Take 
$x \in \g^e_0(j)$. If $j > 0$ 
then $S_{-\gamma}(\pi(\theta(x))) = 
x = S_{-\delta}(x)$. Now assume that $j = 0$, i.e. $x \in \h^e_0$.
As
 $\k_0 = \{0\}$,
we can choose the elements $z_1,\dots,z_{2s}$ 
from Theorem~\ref{lif} so that 
$z_1,\dots,z_s$ (resp.\ $z_{s+1},\dots,z_{2s}$) 
belong to negative (resp.\ positive) $\t^e$-root spaces.
Then
\begin{align*}
\pi(\theta(x))
&=
\pi(x + \textstyle\frac{1}{2}\sum_{i=1}^{2s}
[x,z_i^*]^\ne z_i^\ne)
=\pi(x + \textstyle\frac{1}{2}\sum_{i=1}^{s}
[x,z_i^*]^\ne z_i^\ne)\\
&=
\pi(x + \textstyle\frac{1}{2}\sum_{i=1}^s
z_i^\ne [x,z_i^*]^\ne
+
\frac{1}{2}\sum_{i=1}^s
\langle  [x,z_i^*] | z_i\rangle)\\
&=
x
+
\textstyle\frac{1}{2}\sum_{i=1}^s\langle [x,z_i] | z_i^*\rangle
= x + \gamma(x) - \delta(x),
\end{align*}
noting 
that $\gamma(x) - \delta(x)$ is the trace of
$\frac{1}{2}\ad x$ on $\k_- = \bigoplus_{\alpha \in \Phi^e_-} \k_\alpha$.
Hence $S_{-\gamma}(\pi(\theta(x))) = S_{-\delta}(x)$
as required.
\end{proof}

\begin{Theorem}\label{csa}
The restriction of $S_{-\gamma} \circ \pi:U(\widetilde\p)_0
\twoheadrightarrow U(\p_0)$ defines
a surjective algebra homomorphism
$$\pi_{-\gamma}:U(\g,e)_0 
\twoheadrightarrow
U(\g_0,e)$$
with  $\ker\pi_{-\gamma} = U(\g,e)_{0,\sharp}$.
Hence $
U(\g,e)_0 / U(\g,e)_{0,\sharp} \cong U(\g_0,e)$.
\end{Theorem}

\begin{proof}
Consider the following
diagram:
$$
\begin{CD}
U(\g,e)_0 
\:\:\:{\hookrightarrow}\:\:\:
&U(\widetilde\p)_0
@>S_{-\gamma}\circ \pi>>U(\p_0)&\:\:\:\hookleftarrow \:\:U(\g_0,e)\\
&@V\phi VV@VV\phi_0 V\\
&U(\hat\p)_0
@> \pi>>U(\hat\p_0)\\
&@Vd VV@VV d_0 V\\
&U(\hat\g)_0
@> \pi>>U(\hat\g_0)
\end{CD}
$$
We have already constructed the horizontal maps.
From top to bottom, 
they have kernels $U(\widetilde\p)_{0,\sharp}, U(\widehat\p)_{0,\sharp}$
and $U(\widehat\g)_{0,\sharp}$.
For the vertical maps,
recall the derivation $d:U(\widehat\g)\rightarrow U(\widehat\g)$
and the homomorphism $\phi:U(\widetilde\p) \rightarrow U(\widehat\p)$
from $\S$2. One checks that both of these are $\t^e$-equivariant,
hence they restrict to give the maps $d$ and $\phi$ in the diagram.
The maps $d_0$ and $\phi_0$ come from the analogous maps for the
reductive subalgebra $\g_0$.

We first verify that the top square commutes. 
As $\phi$ is $\t^e$-invariant, it maps
$U(\widetilde\p)_{0,\sharp}$ into $U(\widehat\p)_{0,\sharp}$.
Hence the top square commutes on 
restriction to
$U(\widetilde\p)_{0,\sharp}$, as we get zero both ways round.
Since $U(\widetilde\p)_0 = U(\p)_0 \oplus U(\widetilde\p)_{0,\sharp}$
it remains to check the square commutes on restriction to
$U(\p_0)$. So take $x \in \p_0$.
Introduce the shorthands
$\sum_i^{\pm}$ and $\sum_i^0$
for the sums over
all $1 \leq i \leq r$ such that 
$\be_i|_{\t^e} \in \Phi^e_{\pm}$
and  $\be_i|_{\t^e} = 0$, respectively.
Both ${\textstyle \sum_i^+} f_i[b_i,x]^\ch$ and
${\textstyle \sum_i^-} [b_i,x]^\ch f_i$ belong to
$U(\widehat\p)_{0,\sharp}$, hence map to zero under $\pi$.
So we get
\begin{align*}
\pi(\phi(x)) & = \pi(x+{\textstyle\sum_i} f_i[b_i,x]^\ch)
= \pi(x+{\textstyle \sum_i^0} f_i[b_i,x]^\ch+{\textstyle \sum_i^-} f_i[b_i,x]^\ch)
\\
& = \pi(\phi_0(x)+{\textstyle \sum_i^-} \lan f_i,[b_i,x]\ran+{\textstyle \sum_i^-} [b_i,x]^\ch f_i)\\
&= \phi_0(x)-\textstyle\sum_i^- \langle f_i, [x,b_i] \rangle
= \phi_0(x) - \gamma(x) = \phi_0(S_{-\gamma}(\pi(x))).
\end{align*}

Next we check that the bottom square commutes. 
Again $d$ is $\t^e$-equivariant so maps
$U(\widehat\p)_{0,\sharp}$ into $U(\widehat\g)_{0,\sharp}$,
hence the bottom square commutes on restriction to
$U(\widehat\p)_{0,\sharp}$.
It remains to check it commutes on restriction to
$U(\widehat\p_0)$.
Recalling $\widehat\p_0 = \p_0 \oplus \m^*_0 \oplus \m^\ch_0$, it 
suffices to consider elements $x \in\p_0$, $f \in \m^*_0$ and $y^\ch \in \m^\ch_0$.
In the first case we calculate:
\begin{align*}
\pi(d(x)) &= 
\pi({\textstyle\sum_i^0} f_i [b_i,x]
+
{\textstyle\sum_i^+} f_i [b_i,x]
+
{\textstyle\sum_i^-} f_i[b_i,x])\\
&= 
\pi(d_0(x)
+
{\textstyle\sum_i^-} [b_i,x]f_i)
= d_0(\pi(x)).
\end{align*}
The second case is very similar.
The calculation in the third case is as follows:
\begin{align*}
\pi(d(y^\ch)) 
&=
\pi(y-\chi(y)-y^\ne+\textstyle{\sum_i f_i[b_i,y]^\ch})
=
\pi(d_0(y^\ch)+
\textstyle\sum_i^- f_i[b_i,y]^\ch)
\\
&=
d_0(y^\ch)-
\pi(\textstyle\sum_i^- [b_i,y]^\ch f_i + \sum_i^- \lan f_i, [y,b_i]\ran)
=
d_0(\pi(y^\ch)),
\end{align*}
noting each $\lan f_i, [y,b_i] \ran = 0$ by degree considerations.

Now Theorem~\ref{tidish} gives that $d (\phi(u)) = 0$
for all $u \in U(\g,e)_0$. By
the commutativity of the diagram 
we deduce that $d_0(\phi_0(S_{-\gamma}(\pi(u)))) = 0$ for all such $u$.
By Theorem~\ref{tidish} again, 
we have that
$U(\g_0,e) = \{u \in U({\p}_0)\:|\: d_0(\phi_0(u)) = 0\}$.
Hence $S_{-\gamma}(\pi(U(\g,e)_0)) \subseteq U(\g_0,e)$,
showing that the restriction of $S_{-\gamma}\circ \pi$ defines an algebra 
homomorphism
$\pi_{-\gamma}:U(\g,e)_0 \rightarrow U(\g_0,e)$.
Moreover $U(\g,e)_{0,\sharp} \subseteq
U(\widetilde\p)_{0,\sharp}$ so 
$U(\g,e)_{0,\sharp} \subseteq \ker \pi_{-\gamma}$.
Recall that the quotient $U(\g,e)_0 / U(\g,e)_{0,\sharp}$
has basis given by the cosets of the PBW monomials
$H^{\bf b}$  with $\mathbf b \in \Z_{\geq 0}^l$.
Therefore to complete the proof it suffices to show that
the monomials $\pi_{-\gamma}(H^{\mathbf b}) \in U(\g_0,e)$ actually form
a basis for for
$U(\g_0,e)$, since that shows simultaneously that $\pi_{-\gamma}$ is surjective
and that its kernel is no bigger than $U(\g,e)_{0,\sharp}$.

By Lemma~\ref{beep}, $H^{\mathbf b} = \theta(h^{\mathbf b}) + (\dagger)$
where $h^{\mathbf b} := h_1^{b_1} \cdots h_l^{b_l}$ and $(\dagger)$
denotes a linear combination of terms of strictly smaller degree in the
good grading.
Applying $\pi_{-\gamma}$ using Lemma~\ref{bed}, we deduce that
$\pi_{-\gamma}(H^{\mathbf b})  = S_{-\delta}(h^{\mathbf b}) + (\ddagger)$
where $(\ddagger)$ consists of lower degree terms.
Recalling that $\gr' U(\g_0,e) = U(\g^e_0)$ by Theorem~\ref{TGoodFil},
we see by the PBW theorem for $U(\g^e_0)$ that the monomials $\pi_{-\gamma}(H^{\mathbf b})$ for all $\mathbf b \in \Z_{\geq 0}^l$
do indeed form a basis for $U(\g_0,e)$.
\end{proof}

\subsection{Verma modules}\label{sVerma}
Recall the embedding
$\theta:\t^e \hookrightarrow U(\g,e)$
from Theorem~\ref{lif}
and the weight $\delta$ from $\S$\ref{sCartan}.
For a $U(\g,e)$-module $V$ and
$\la\in(\t^e)^*$, we define the {\em $\la$-weight space}
$$
V_\la:=\{v\in V\:|\: (\theta+\delta)(t) v = \la(t) v\text{ for all } t \in \t^e\}.
$$
By Theorem~\ref{lif}(1) we have that $U(\g,e)_\al V_\la\subseteq V_{\la+\al}$. 
In particular each $V_\la$ is invariant under the action of the subalgebra
$U(\g,e)_0$.
We say that $V_\la$ is a {\em maximal weight space} of $V$
if $U(\g,e)_\sharp V_\la = \{0\}$.
For example, if $\la$ is any {\em maximal weight} of $V$ in the dominance
ordering, i.e. $V_\la \neq \{0\}$ and $V_\mu = \{0\}$ for all $\mu > \la$,
then $V_\la$ is a maximal weight space of $V$.

Let $V_\la$ be a maximal weight space in
a $U(\g,e)$-module $V$.
Then the action of $U(\g,e)_0$ on 
$V_\la$ factors through the map $\pi_{-\gamma}$ from Theorem~\ref{csa}
to make
$V_\la$ into a $U(\g_0,e)$-module
such that $u m = \pi_{-\gamma}(u) m$ for $u \in U(\g,e)_0$ and $m \in V_\la$.
Note also that $\t^e$ is a Lie subalgebra of $U(\g_0,e)$
(since $\t^e$ even lies in the center of $U(\p_0)$),
so we get another action of $\t^e$ on $V_\la$
by restricting the $U(\g_0,e)$-action.
By Lemma~\ref{bed} this new action satisfies
$$
t v = \la(t) v
$$
for all $t \in \t^e$.
This explains why we included the additional shift by $\delta$
in the  definition of the $\la$-weight space  of a $U(\g,e)$-module.

We say that a $U(\g,e)$-module $V$
is a {\em highest weight module} 
if it is generated 
by a maximal weight space $V_\la$
such that $V_\la$ is finite 
dimensional and irreducible as a $U(\g_0,e)$-module.
In that case, as we will see shortly, $\la$ is the unique maximal
weight of $V$ in the dominance ordering.
Let
$$
\{V_\La\:|\:\La \in \mathcal L\}
$$
be a complete set of pairwise inequivalent 
finite dimensional irreducible $U(\g_0,e)$-modules
for some set $\mathcal L$.
If $V$ is a highest weight module generated by
a maximal weight space $V_\la$ and $V_\la \cong V_\La$
for $\La \in \mathcal L$,
we say that $V$ is 
of {\em type} $\La$.

Since $U(\g,e)_\sharp$ is invariant under left multiplication
by $U(\g,e)$ and right multiplication by $U(\g,e)_0$,
the quotient $U(\g,e) / U(\g,e)_\sharp$ is
a $(U(\g,e), U(\g,e)_0)$-bimodule.
Moreover the right action of $U(\g,e)_0$ factors through the
map $\pi_{-\gamma}$ from Theorem~\ref{csa} to make
$U(\g,e) / U(\g,e)_\sharp$ into a $(U(\g,e), U(\g_0,e))$-bimodule.
For $\La\in{\mathcal L}$ we define the {\em Verma module of type $\La$} by
setting
$$
M(\La,e):=(U(\g,e)/U(\g,e)_\sharp)\otimes_{U(\g_0,e)} V_\La.
$$
We are going show that this is a universal
highest weight module of type $\La$, meaning that
 $M(\La,e)$ is a highest weight module of type $\La$ (in particular
it is non-zero) and moreover
if $V$ is another highest weight module 
generated by a maximal weight space $V_\la$
and $f: V_\La \stackrel{\sim}{\rightarrow} V_\la$
is a $U(\g_0,e)$-module isomorphism
then there is a unique $U(\g,e)$-module
homomorphism $\tilde{f}: M(\La,e)\twoheadrightarrow V$ extending $f$.
Recall the PBW basis for $U(\g,e)$ fixed in $\S$\ref{sCartan}.

\begin{Lemma}\label{LRBasis}
As a right $U(\g_0,e)$-module,
$U(\g,e)/U(\g,e)_\sharp$ is free
with basis $\{F^\aa\:|\:\aa\in\Z_{\geq 0}^m\}.$
\end{Lemma}

\begin{proof}
This follows because
the cosets of the PBW monomials of
the form $F^\aa H^\bb$ give a basis for the quotient
$U(\g,e)/U(\g,e)_\sharp$
and the cosets of the monomials of the form $H^\bb$
give a basis for  $U(\g,e)_0/U(\g,e)_{0,\sharp} \cong U(\g_0,e)$.
\end{proof}

\begin{Theorem} \label{Tverma}
Given $\La\in{\mathcal L}$,
let $v_1,\dots,v_k$ be a basis for $V_\La$
and $\la$ be its $\t^e$-weight.
\begin{enumerate}
\item[{\rm (1)}] 
The vectors $\{F^{\mathbf a} \otimes  v_i\:|\:\mathbf a \in \Z^m_{\geq 0},
1 \leq i \leq k\}$ give a basis of $M(\La,e)$.
\item[{\rm (2)}] The weight $\la$ is the unique maximal weight of $M(\La,e)$
in the dominance ordering,
$M(\La,e)$ is generated by the maximal weight space
$M(\La,e)_\la$,
and $M(\La,e)_\la \cong V_\La$ as $U(\g_0,e)$-modules.
\item[{\rm (3)}] The module $M(\La,e)$ is a universal highest weight module of
type $\La$.
\item[{\rm (4)}] There is a unique maximal proper submodule
$R(\La,e)$ in 
$M(\La,e)$,
$$
L(\La,e):=M(\La,e)/R(\La,e)
$$
is irreducible, and
$
\{L(\La,e)\:|\: \La\in {\mathcal L}\}$
is a complete set of pairwise inequivalent 
irreducible highest weight modules over $U(\g,e)$.
\end{enumerate}
\end{Theorem}

\begin{proof}
(1)
This is clear from Lemma~\ref{LRBasis}.

(2)
The basis vector $F^{\mathbf a} \otimes v_i$
is of weight $\la - \sum_i a_i \ga_i$.
Hence the $\la$-weight space of $M(\La,e)$ is $1 \otimes V_\La$
and all other weights of $M(\La,e)$ are strictly smaller in the dominance
ordering. 

(3) By (1)--(2) $M(\La,e)$ is a highest weight module of type $\La$.
Let $V$ be another highest weight module generated
by a maximal weight space $V_\mu$ and $f:V_\La \rightarrow V_\mu$
be a $U(\g_0,e)$-module isomorphism. By comparing $\t^e$-actions
we get that $\mu = \la$.
By adjointness of tensor and hom
$f$ extends uniquely to a $U(\g,e)$-module homomorphism
$\tilde f: M(\La,e) \rightarrow V$
such that $\tilde f(1 \otimes v_i) = f(v_i)$ for each $i$. 
As $\tilde f(1 \otimes V_\La) = f(V_\La)$ 
generates $V$, we get that $\tilde f$ is surjective.

(4) Let $N$ be a submodule of $M(\La,e)$.
Then $N$ is the direct sum of its $\t^e$-weight spaces.
If $N_\la \neq 0$ then $N_\la$ generates all of 
$1 \otimes V_\La$ as a $U(\g_0,e)$-module, 
hence it generates all of $M(\La,e)$ as a $U(\g,e)$-module.
This shows that if $N$ is a proper submodule then 
it is contained in $\bigoplus_{\mu < \la} M(\La,e)_\mu$.
Hence the sum of all proper submodules of $M(\La,e)$ is still a proper
submodule, so $M(\La,e)$ has a unique maximal submodule $R(\La,e)$ as claimed.
By (3) any irreducible highest weight module $V$ of type $\La$
is a quotient of $M(\La,e)$, hence 
$V \cong 
L(\La,e)$. Moreover $\la$ is the unique maximal weight of 
$L(\La,e)$ by (2) and $L(\La,e)_\la \cong V_\La$ as $U(\g_0,e)$-modules.
Hence $\La$ is uniquely determined by $V$.
\end{proof}

\begin{Corollary}\label{nouse}
Let $\mathcal L^+ := \{\La \in \mathcal L\:|\:
\dim L(\La,e) < \infty\}$. Then the modules
$\{L(\La,e)\:|\:\La \in \mathcal L^+\}$ give a complete set of
pairwise inequivalent finite dimensional
irreducible $U(\g,e)$-modules.
\end{Corollary}

\begin{proof}
Any finite dimensional irreducible $U(\g,e)$-module $L$
has a maximal weight $\la$.
Moreover by irreducibility $L$ is generated by any
irreducible $U(\g_0,e)$-submodule of $L_\la$. Hence $L$
is an irreducible highest weight module. Now apply Theorem~\ref{Tverma}(4).
\end{proof}

Unfortunately we have absolutely no idea how to give
an explicit combinatorial
parametrization
$\{V_\La\:|\:\La \in \mathcal L\}$
of a complete set of pairwise inequivalent
finite dimensional irreducible
$U(\g_0,e)$-modules 
unless the distinguished nilpotent element $e$ of $\g_0$
is actually regular.
We will discuss the regular case in more detail in $\S$5.1.

\subsection{Central characters}\label{scharacters}
Let $Z(\g)$ denote the center of $U(\g)$ and
$Z(\g,e)$ denote the center of $U(\g,e)$.
It is easy to see that the restriction of the linear map
$\Pr$ from $\S$2
defines an injective algebra homomorphism
$\Pr:Z(\g) \hookrightarrow Z(\g,e)$.
As explained in 
the footnote to \cite[Question 5.1]{P2},
this map is also surjective, so it is an
algebra isomorphism
$$
\Pr:Z(\g) \stackrel{\sim}{\rightarrow} Z(\g,e).
$$
A $U(\g,e)$-module $V$ is of
{\em central character} $\psi:Z(\g) \rightarrow \C$ if 
$\Pr(z) v = \psi(z) v$ for all $z \in Z(\g)$ and $v \in V$.
Analogously for $\g_0$ we have the isomorphism
$$
\Pr_0:Z(\g_0) \stackrel{\sim}{\rightarrow} Z(\g_0,e),
$$
and we say that a $U(\g_0,e)$-module $V$ is of
central character $\psi_0:Z(\g_0) \rightarrow \C$
if $\Pr_0(z) v = \psi_0(v)$ for all $z \in Z(\g_0)$ and $v \in V$.
We want to relate the central character of an irreducible
highest weight module over $U(\g,e)$  to
the central character of its maximal weight space over
$U(\g_0,e)$. We remark that
the surjectivity of the map $\Pr$ will not be used in any
of the arguments below 
but the surjectivity of $\Pr_0$ is essential;
if $e$ is regular in $\g_0$ then the surjectivity of $\Pr_0$
is already clear from \cite[$\S$2]{Kostant}.

Let $\Phi$ (resp.\ $\Phi_0$) denote the root system of $\g$ 
(resp.\ $\g_0$) with respect to $\t$
and $W$ (resp.\ $W_0$) be the corresponding Weyl group.
Let $\Phi_{\pm} := \left\{\alpha \in \Phi\:\big|\:\alpha|_{\t^e} \in \Phi^e_{\pm}\right\}$
so that 
$$
\Phi = \Phi_- \sqcup \Phi_0 \sqcup \Phi_+,
$$
corresponding to the decomposition $\g = \g_- \oplus \g_0 \oplus \g_+$.
We stress that $\Phi_+$ is {\em not} a positive system of roots in $\Phi$;
we reserve the notation $\Phi^+$ for such a thing below.
For each $\alpha \in \Phi$, fix a non-zero vector
$x_\alpha$ in the $\alpha$-root space of $\g$
and let $\alpha^\vee := 2 \alpha / (\alpha|\alpha) \in \t^*$.

Now we define Harish-Chandra isomorphisms
$$
\Psi:Z(\g) \stackrel{\sim}{\rightarrow} S(\t)^W,
\qquad
\Psi_0:Z(\g_0)\stackrel{\sim}{\rightarrow} S(\t)^{W_0}
$$
for $\g$ and $\g_0$ as follows.
Let $\Phi^+$ be any system of positive roots in $\Phi$ and 
set $\Phi_0^+ := \Phi^+ \cap \Phi_0$, which is 
a system of positive roots
in $\Phi_0$.
Set
$$
\rho := {\textstyle\frac{1}{2}}\sum_{\alpha \in \Phi^+} \alpha,
\qquad
\rho_0 :=
{\textstyle\frac{1}{2}}
\sum_{\alpha \in \Phi_+} \alpha
+ 
 {\textstyle\frac{1}{2}}
\sum_{\alpha \in \Phi_0^+} \alpha.
$$
The first term on the right hand side of the
definition of $\rho_0$ is orthogonal to all the roots in 
$\Phi_0$; it should be viewed as 
a shift in origin for the definition of $\Psi_0$.
Then, the Harish-Chandra isomorphisms
$\Psi$ and $\Psi_0$ are determined on $z \in Z(\g)$ and $z_0 \in Z(\g_0)$
by the equations
\begin{align*}
z &\equiv S_\rho(\Psi(z)) \:\;\;\pmod{\textstyle\sum_{\alpha \in \Phi^+} U(\g) x_\alpha},
\\
z_0 &\equiv S_{\rho_0}(\Psi_0(z)) \pmod{\textstyle\sum_{\alpha \in \Phi_0^+} U(\g_0) x_\alpha}.
\end{align*}
Although the definitions of $\Psi$ and $\Psi_0$ 
involve the choice of $\Phi^+$,
it is known by \cite[7.4.5]{Dix} that 
the Harish-Chandra isomorphisms $\Psi$ and $\Psi_0$ 
are actually independent of this choice
(though $\Psi_0$ does depend on the choice of $\Phi^e_+$ because of the
shift of origin in defining $\rho_0$).

\begin{Theorem}\label{ed}
There is a unique embedding
$c:Z(\g) \hookrightarrow Z(\g_0)$ such that
the following diagram commutes:
$$
\begin{CD}
Z(\g,e)@<\Pr <<Z(\g)@>\Psi >>S(\t)^W\\
@V\pi_{-\gamma} VV@VVc V@VV\iota V\\
Z(\g_0,e)@<\Pr_0 <<Z(\g_0)@>\Psi_0 >>S(\t)^{W_0}
\end{CD}
$$
Here,
 $\iota:S(\t)^W \hookrightarrow S(\t)^{W_0}$ 
denotes the natural inclusion.
\end{Theorem}

\begin{proof}
Given $\sigma,\tau \in \{+,0,-\}$ let
$\Phi_\sigma(\tau)$ denote the set of all $\alpha \in \Phi_\sigma$
such that the degree of $x_\alpha$ in the good grading
is positive, zero or negative according to whether
$\tau = +, 0$ or $-$, respectively.
Note that $\Phi_\sigma(\tau) = - \Phi_{-\sigma}(-\tau)$,
in particular, $\Phi_0(0)$ is a closed subsystem of $\Phi$ (it is the root
system of the reductive Lie algebra $\h_0$).
Pick a system $\Phi^+_0(0)$ of positive
roots in $\Phi_0(0)$.
Then set
\begin{align*}
\Phi^+ &:= \Phi^+_0(0)\sqcup \Phi_+(0) \sqcup \Phi_+(-) \sqcup \Phi_0(-)
\sqcup \Phi_-(-),
\end{align*}
which is a system of positive roots in
$\Phi$.
Using this choice we define the weights
$\rho$, $\rho_0$ and the 
Harish-Chandra isomorphisms
$\Psi$, $\Psi_0$, as explained above.
So
\begin{align*}
\rho &= \textstyle
{\textstyle\frac{1}{2}} \sum
\left( \Phi^+_0(0)
\sqcup \Phi_+(0) \sqcup \Phi_+(-)
\sqcup \Phi_0(-)
\sqcup \Phi_-(-)\right),\\
\rho_0 &=\textstyle {\textstyle\frac{1}{2}} \sum 
\left(\Phi_+(+)
\sqcup \Phi_+(0)
\sqcup \Phi_+(-)
\sqcup \Phi^+_0(0)\sqcup \Phi_0(-)\right).
\end{align*}
Recalling the weight $\gamma$ from $\S$\ref{sCartan},
we deduce that
$$
\textstyle\rho - \rho_0 = {\textstyle\frac{1}{2}} \sum \Phi_-(-) - 
{\textstyle\frac{1}{2}}\sum \Phi_+(+)
=\sum \Phi_-(-) = \gamma.
$$

Now we need to fix some PBW bases.
Let $t_1,\dots,t_n$ be a basis for $\t$.
Also enumerate the elements of
$\Phi_0^+(0)$ as $\alpha_1,\dots,\alpha_s$,
the elements of $\Phi_0(-)$ as $\beta_1,\dots,\beta_q$,
the elements of $\Phi_+(-)\sqcup \Phi_-(-)$ as
$\beta_{q+1},\dots,\beta_r$,
and the elements of $\Phi_+(0)$ as $\nu_1,\dots,\nu_p$.
Order the basis
$\{t_i, x_\alpha\:|\:1 \leq i \leq n, \alpha \in \Phi\}$ of $\g$
so the $x_{-\beta_i}$ come first, then
the $x_{-\nu_i}$, then the $x_{-\alpha_i}$,
then the $t_i$, then the $x_{\alpha_i}$, then
the $x_{\nu_i}$, then the $x_{\beta_i}$.
This ordering determines a corresponding PBW basis for $U(\g)$.
Using the same ordering on the basis
$\{t_i,x_\alpha\:|\:1 \leq i \leq n, \alpha\in \Phi_0\}$ of $\g_0$
we also get a PBW basis for $U(\g_0)$.

We next calculate $\Pr_0(z_0)$
for any $z_0 \in Z(\g_0)$.
We can write
$$
z_0 = S_{\rho_0}(\Psi_0(z_0)) + u_0 + v_0
$$
where $S_{\rho_0}(\Psi_0(z_0)) \in  S(\t)$, $u_0$ 
is a linear combination of ordered
PBW monomials ending in  $x_{\alpha_i}\:(1 \leq i \leq s)$
and $v_0$ is a linear combination of ordered PBW monomials
ending in $x_{\beta_i}\:(1 \leq i \leq q)$.
As $z_0$ is central it is of degree $0$ in the good grading.
Hence $S_{\rho_0}(\Psi_0(z_0)), u_0$ and $v_0$ are all of degree $0$ too.
Note $u_0$  only involves products of the basis vectors 
$x_{\alpha_1},\dots,x_{\alpha_s}$ 
on the positive side, all of which are of degree $0$. 
Hence it can only involve products of the basis vectors
$x_{-\alpha_1},\dots,x_{-\alpha_s}$ on the negative side because
each $x_{-\beta_1},\dots,x_{-\beta_q}$ is of strictly positive degree.
Hence $u_0 \in U(\h_0)$ and $\Pr_0(S_{\rho_0}(\Psi_0(z_0)) + u_0) = 
S_{\rho_0}(\Psi_0(z_0)) + u_0$.
To compute $\Pr_0(v_0)$, 
our PBW monomials are ordered so that the basis vectors 
$x_{\beta_1},\dots,x_{\beta_q}$ for $\m_0$ appear on the right hand side,
so we simply replace each $x_{\beta_i}$ by the scalar $\chi(x_{\beta_i})$.
Since all $x_{\beta_i}$ are of strictly negative degree, it follows that
$\Pr_0(v_0)$ is of strictly positive degree.
We have shown that
$$
\Pr_0(z_0) = S_{\rho_0}(\Psi_0(z_0)) + u_0 + \Pr_0(v_0) \in
Z(\g_0,e) \subseteq U(\p_0)
$$
with $S_{\rho_0}(\Psi_0(z_0)) \in S(\t)$,
$u_0 \in \sum_{i=1}^s U(\h_0) x_{\alpha_i}$
and $\Pr_0(v_0) \in
\sum_{j > 0}
U(\p_0)(j)$.
In particular we see from this that 
$z_0$ can be recovered
uniquely from $\Pr_0(z_0)$: it is the unique
element of $Z(\g_0)$ such that
$$
S_{\rho_0}(\Psi_0(z_0)) \equiv \Pr_0(z_0) \pmod{\textstyle
\sum_{i=1}^s U(\h_0) x_{\alpha_i} + \sum_{j > 0} U(\p_0)(j)}.
$$

Instead take an element $z \in Z(\g)$.
We expand it as
$$
z = S_\rho(\Psi(z)) + t+ u + v
$$
where $S_\rho(\Psi(z)) \in S(\t)$, 
$t$ is a linear combination of ordered PBW monomials
ending in $x_{\alpha_i}\:(1 \leq i \leq s)$,
$u$ is a linear combination of ordered PBW monomials
ending in $x_{\nu_i}\:(1 \leq i \leq p)$,
and $v$ is a linear combination of ordered PBW monomials
ending in $x_{\beta_i}\:(1 \leq i \leq r)$.
We apply
the map $\Pr$ and argue just like in the previous paragraph
to get that
$$
\Pr(z) = S_\rho(\Psi(z)) + t+u + \Pr(v) \in Z(\g,e) \subseteq U(\widetilde\p)_0
$$
where $S_{\rho}(\Psi(z)) \in S(\t)$,
$t \in \sum_{i=1}^s U(\h) x_{\alpha_i}$, $u \in \sum_{i=1}^p U(\h) x_{\nu_i}$
and finally $\Pr(v) \in
U(\widetilde\p) \k^\ne+\sum_{j > 0}
U(\p)(j)$.
Next we apply the graded homomorphism $S_{-\gamma}\circ\pi:U(\widetilde\p)_0
\rightarrow U(\p_0)$ to this. It annihilates 
$U(\h) x_{\nu_i}$ and $U(\widetilde\p)_0 \cap U(\widetilde\p) \k^\ne$.
So we deduce recalling Theorem~\ref{csa} that
$$
\pi_{-\gamma}(\Pr(z)) = S_{-\gamma}(S_{\rho}((\Psi(z))) + w \in Z(\g_0,e)
$$
for $w \in \sum_{i=1}^s U(\h_0) x_{\alpha_i}
+ \sum_{j > 0} U(\p_0)(j)$.
Using the last sentence of the previous paragraph 
and the fact from 
\cite{P2}
that $\Pr_0: Z(\g_0) \hookrightarrow Z(\g_0,e)$ is surjective, 
we deduce that
$\pi_{-\gamma}(\Pr(z)) = \Pr_0(z_0)$
where $z_0$ is the unique element of $Z(\g_0)$
such that $S_{\rho_0}(\Psi_0(z_0)) = S_{-\gamma}(S_\rho(\Psi(z)))$.
Equivalently, by the first paragraph, $\Psi_0(z_0) = \Psi(z)$.

Now we can prove the theorem.
Since $\Psi$ and $\Psi_0$ are isomorphisms, there is obviously
a unique map $c$ so the right hand square commutes.
For the left hand square, we have shown 
for $z \in Z(\g)$ that
$\pi_{-\gamma}(\Pr(z)) = \Pr_0(z_0)$ where
$\Psi_0(z_0) = \Psi(z)$. This means that $z_0 = c(z)$
hence $\pi_{-\gamma}(\Pr(z)) = \Pr_0(c(z))$.
\end{proof}

For $\La \in \mathcal L$, Schur's lemma implies that
$Z(\g_0,e)$ acts diagonally on $\La$.
Let $\psi^\La_0:Z(\g_0) \rightarrow \C$ be the resulting central character,
i.e. $\Pr_0(z) v = \psi^\La_0(z) v$ for all $z \in Z(\g_0)$ and $v \in \La$.
Let $\psi^\La:Z(\g) \rightarrow \C$ denote $\psi^\La_{0} \circ c$.

\begin{Corollary}\label{hascc}
Let $V$ be a highest weight $U(\g,e)$-module of type $\La \in \mathcal L$.
Then $V$ is of central character $\psi^\La: Z(\g) \rightarrow \C$.
\end{Corollary}

\begin{proof}
Say $V$ is generated by its maximal weight space $V_\la$
and $f:V_\La \rightarrow V_\la$ is an isomorphism
of $U(\g_0,e)$-modules.
For $z \in Z(\g)$ and $v \in V_\La$, the theorem implies that
\begin{align*}
\Pr(z)f(v) &= f(\pi_{-\gamma}(\Pr(z)) v) = 
f(\Pr_0(c(z)) v)= \psi_0^\La(c(z)) f(v).
\end{align*}
Hence $\Pr(z)$ acts on $V_\la$ as the scalar
$\psi^\La(z)$. Since $V_\la$ generates $V$ as
a $U(\g,e)$-module it follows that $\Pr(z)$ acts on all of $V$
as $\psi^\La(z)$. Hence $V$ has central character
$\psi^\La$.
\end{proof}

\begin{Remark}\label{vary}\rm
Given in addition a weight $o \in \t^*$
orthogonal to all the roots in
$\Phi$, one can modify the above definitions of $\rho$ and
$\rho_0$ by adding $o \in \t^*$ to them both (``change of origin''). Providing
one also adds $o$ to the weight $\epsilon$ defined in
$\S$\ref{sLevi} below, all our subsequent results remain true as stated.
The point is that although adding $o$
changes the Harish-Chandra isomorphisms $\Psi$ and
$\Psi_0$, hence also the parametrization of
central characters, it does not affect the
maps $c$ or $\iota$ in Theorem~\ref{ed}.
\end{Remark}

\subsection{\boldmath Category $\mathcal O(e)$}\label{sO}
We want to prove that Verma modules have finite length.
This is not a hard result if $e$ is regular in $\g_0$, but
to prove it in general we need to appeal to some recent results of Losev.
We first recall a little more of the background for this.
Let $\lag$ be a Lagrangian subspace of $\k$ (for example
a natural choice is
$\lag := \k_+ = \bigoplus_{\alpha \in \Phi^e_+} \k_\alpha$).
Define the left $U(\g)$-module $Q_\lag = U(\g) / I_\lag$ and the algebra
$Q_\lag^{\m\oplus\lag}$ as in Remark~\ref{premets}.
It is obvious from the definition of $Q_\lag^{\m\oplus\lag}$
that there is a well-defined 
multiplication map
$$
Q_\lag \otimes Q_\lag^{\m\oplus\lag} \rightarrow Q_\lag,\quad
(u+I_\lag)\otimes (v+I_\lag) \mapsto uv+I_\lag
$$
making $Q_\lag$ into a $(U(\g), Q_\lag^{\m\oplus\lag})$-bimodule.
Identifying $Q_\lag^{\m\oplus\lag}$ with
$U(\g,e)$ using the isomorphism $\nu$ from Remark~\ref{premets}
and the isomorphism from Theorem~\ref{t1}, this makes $Q_\lag$ into
a $(U(\g), U(\g,e))$-bimodule too.

Let $\mathcal C(e)$ denote the category of all left
$U(\g,e)$-modules.
Let $\mathcal W(e)$ denote the category of all $\g$-modules
on which $x - \chi(x)$ acts locally nilpotently for all $x \in \m\oplus\lag$.
Note $Q_\lag$ belongs to $\mathcal W(e)$, hence tensoring with
this bimodule defines a functor
$$
Q_\lag \otimes_{U(\g,e)}? : \mathcal C(e)
\rightarrow \mathcal W(e).
$$
The important {\em Skryabin's theorem} asserts that
this functor is an equivalence of categories;
see \cite{Skryabin} or \cite[Theorem 6.1]{GG}.

Hence if $L$ is any irreducible $U(\g,e)$-module then
$Q_\lag \otimes_{U(\g,e)} L$ is an irreducible $U(\g)$-module,
and its annihilator
$\ann_{U(\g)}(Q_\lag \otimes_{U(\g,e)} L)$
is a primitive ideal of $U(\g)$.
For any primitive ideal $P$ of $U(\g)$, we let
$\VA(P) \subseteq \g$ denote its associated variety;
see e.g. \cite[$\S$9]{Ja2}.
It is known that
$\VA(P)$ is the closure of a single nilpotent orbit in $\g$;
see \cite[3.10]{Joseph}.
By \cite[Theorem 3.1]{P2} and \cite[Theorem 1.2.2(ii),(ix)]{Losev}, it is known for any irreducible $U(\g,e)$-module $L$ that
$$
\VA(\ann_{U(\g)}(Q_\lag \otimes_{U(\g,e)} L))
 \supseteq \overline{G \cdot e}
$$
with equality if and only if $L$ is finite dimensional.

\begin{Theorem}\label{finiteness}
The number of isomorphism classes of irreducible highest weight
modules for $U(\g,e)$ with prescribed central character $\psi:Z(\g) \rightarrow \C$ 
is finite, i.e. the set
$\{\La \in \mathcal L\:|\:\psi^\La = \psi\}$ 
is finite.
\end{Theorem}

\begin{proof}
By Corollary~\ref{hascc},
 $\psi^\La = \psi$
implies $\psi_0^\La = \psi_0$ for
some central character $\psi_0:Z(\g_0) \rightarrow \C$ such that $\psi_0\circ c = \psi$.
Each $W$-orbit in $\t^*$ is a union of finitely many
$W_0$-orbits, hence there are finitely many such $\psi_0$.
Therefore it suffices to prove for each $\psi_0:Z(\g_0) \rightarrow \C$
that the set
$\{\La \in \mathcal L\:|\:\psi_0^\La = \psi_0\}$
is finite.
In other words, replacing $\g$ by $\g_0$, we may assume that $e$ is
a distinguished nilpotent element in $\g$ 
and need to prove that the number of isomorphism classes
of {\em finite dimensional} irreducible $U(\g,e)$-modules with
fixed central character $\psi$ is finite. This statement is immediate
if $e$ is regular in $\g$ by \cite[$\S$2]{Kostant}.
In general we use \cite[Theorem 1.2.2]{Losev},
as follows.
The map sending $L$ to $\ann_{U(\g,e)}(L)$ induces a bijection between
isomorphism classes of finite dimensional irreducible $U(\g,e)$-modules
of central character $\psi$
and primitive ideals of $U(\g,e)$ of finite codimension
that contain $\Pr(\ker\psi)$.
So we just need to show there are finitely many such primitive 
ideals.
By \cite[Theorem 1.2.2(ii),(iii)]{Losev},
if $P = \ann_{U(\g,e)}(L)$ 
is a primitive ideal of $U(\g,e)$ containing $\Pr(\ker\psi)$
then $\ann_{U(\g)}(Q_\lag \otimes_{U(\g,e)} L)$ is a primitive
ideal of $U(\g)$ containing $\ker \psi$.
A well known consequence of Duflo's theorem \cite{Duflo}
is that there are only finitely many such primitive ideals
in $U(\g)$. Hence there are only finitely many possibilities for $P$
thanks to \cite[Theorem 1.2.2(vi),(vii)]{Losev}.
\end{proof}

\begin{Corollary}\label{flen}
For each $\La \in \mathcal L$, the Verma module $M(\La,e)$ has 
a composition series.
\end{Corollary}

\begin{proof}
Imitate the standard argument in the classical case from
\cite[7.6.1]{Dix}, using Corollary~\ref{hascc},
Theorem~\ref{Tverma}(1)--(2) and Theorem~\ref{finiteness}.
\end{proof}

Now we introduce an analogue of the Bernstein-Gelfand-Gelfand
category $\mathcal O$:
let $\mathcal O(e) = \mathcal O(e;\t,\q)$ 
denote the category of all finitely
generated $U(\g,e)$-modules $V$ that are semisimple over $\t^e$
with finite dimensional $\t^e$-weight spaces,
such that the set 
$\{\la \in (\t^e)^*\:|\:V_\la \neq \{0\}\}$
 is contained in a finite union of sets of the
form $\{\nu\in(\t^e)^*\:|\: \nu\leq\mu\}$ for $\mu \in (\t^e)^*$.
As $U(\g,e)$ is Noetherian, $\mathcal O(e)$ is closed under 
the operations of taking submodules, quotients and finite direct sums.
The following statement follows routinely
from Corollary~\ref{flen}
and Theorem~\ref{Tverma}.

\begin{Corollary}
Every object in $\mathcal O(e)$ has a composition series.
Moreover the category $\mathcal O(e)$ decomposes as
$\mathcal O(e) = \bigoplus_{\psi}  \mathcal O_\psi(e)$,
where the direct sum is over all central characters $\psi:Z(\g) \rightarrow \C$,
and $\mathcal O_\psi(e)$ denotes the Serre subcategory of $\mathcal O(e)$
generated by the irreducible modules
$\{L(\La,e)\:|\:\La \in \mathcal L\text{ such that }\psi^\La = \psi\}$.
\end{Corollary}

In particular, this shows that the irreducible objects in $\mathcal O(e)$
are all of the form $L(\La,e)$ for $\La \in \mathcal L$.
In the case $e = 0$,
$\mathcal O(e)$ 
is the usual BGG category $\mathcal O$ for the semisimple Lie
algebra $\g$ with respect to the maximal toral subalgebra
$\t$ and the Borel $\q$. At the other extreme, if $e$ is a distinguished 
nilpotent element
of $\g$, then $\mathcal O(e)$ is 
the category of all finite dimensional $U(\g,e)$-modules that are semisimple
over the Lie algebra center of $\g$.

\begin{Remark}\rm
If $e$ is a distinguished but not a regular nilpotent 
element of $\g$ then
$U(\g,e)$ has primitive ideals of infinite codimension 
by \cite[Theorem 1.2.2(viii)]{Losev} and \cite[Theorem 3.1]{P2}. So
there is no chance in this case 
that every primitive ideal of $U(\g,e)$ is the annihilator
of an irreducible highest weight module in our narrow sense (finite dimensional weight spaces). 
\end{Remark}

\section{Special cases}

In this section we specialize further.
First we discuss the case that 
$e$ is of {\em standard Levi type} in the sense of \cite{FP}, i.e. 
it is a regular nilpotent element of $\g_0$. In particular, we will
formulate a precise conjecture for the
classification of finite dimensional irreducible $U(\g,e)$-modules
in standard Levi type.
Then we prove this conjecture for the standard choice of positive
roots in type A, by
translating some
results from \cite{BK} into the present framework.
We continue with the notation from the previous section; in
particular, recall we fixed a parabolic subalgebra $\q$ with Levi factor $\g_0$ 
in $\S$\ref{sCartan}.

\subsection{Standard Levi type}\label{sLevi}
Assume from now on that $e$ is a regular nilpotent element of $\g_0$.
In that case,
by \cite[$\S$2]{Kostant}, the map $\Pr_0$ is an isomorphism
$$
\Pr_0:Z(\g_0) \stackrel{\sim}{\rightarrow} U(\g_0,e).
$$
As we have already observed in $\S\S$\ref{scharacters}--\ref{sO}, many things in the theory are simpler under this assumption.
To start with, $\p_0$
is actually a Borel subalgebra of $\g_0$ with opposite nilradical $\n_0$.
Let $\Psi$ and $\Psi_0$ be the Harish-Chandra isomorphisms for $\g$
and $\g_0$, defined as in $\S$\ref{scharacters}. 
Recalling that
$\Phi_{+}=\left\{\alpha \in \Phi\:\big|\:\alpha|_{\t^e} \in \Phi^e_{+}\right\}$ 
is the set
of roots corresponding to the nilradical $\g_+$ of $\q$,
let
$$
\epsilon := {\textstyle{\frac{1}{2}}} \sum_{\alpha \in \Phi_+} \alpha
+
{\textstyle{\frac{1}{2}}} \!\sum_{\substack{1 \leq i \leq r \\ \beta_i|_{\t^e} = 0}}
\beta_i.
$$
This is just the weight $\rho_0$ from $\S$\ref{scharacters} 
for the system of positive roots in $\Phi_0$ 
corresponding to the Borel subalgebra $\t \oplus \n_0$ of $\g_0$ (though 
we don't necessarily want to fix this choice).
With this in mind, the following lemma is essentially 
\cite[Proposition 2.3]{Kostant}:

\begin{Lemma}\label{ks}
Let $\xi:U(\p_0) \twoheadrightarrow S(\t)$ 
be the homomorphism induced by the natural projection 
$\p_0 \twoheadrightarrow \t$. 
Then the restriction of $S_{-\epsilon} \circ \xi$ defines an algebra isomorphism
$$
\xi_{-\epsilon}: U(\g_0,e) \stackrel{\sim}{\rightarrow} S(\t)^{W_0}
$$
such that
$\Psi_0 = \xi_{-\epsilon} \circ \Pr_0$.
\end{Lemma}

We can now define an explicit set $\mathcal L$ parametrizing the Verma modules
$M(\La,e)$ and the irreducible highest weight modules $L(\La,e)$
for $U(\g,e)$:
define 
$$
\mathcal L := \t^* / W_0 = \operatorname{Spec}(S(\t)^{W_0}).
$$ 
Thus each $\La \in \mathcal L$ is a $W_0$-orbit of weights in $\t^*$.
For each $\La \in \mathcal L$, let 
$V_\La$ denote the
one dimensional irreducible $U(\g_0,e)$-module 
obtained  by lifting the irreducible $S(\t)^{W_0}$-module corresponding to
$\La$ through the isomorphism
$\xi_{-\epsilon}:U(\g_0,e) \stackrel{\sim}{\rightarrow} S(\t)^{W_0}$
from Lemma~\ref{ks}.
Also fix finally a Borel subalgebra $\b_0$ of $\g_0$ containing 
$\t$ and let
$$
\b := \b_0 \oplus \g_+,
$$
which is a Borel subalgebra of $\g$ contained
in the parabolic $\q$.
Let $\Phi_0^+$ and $\Phi^+ = \Phi_0^+ \sqcup \Phi_+$ be the 
systems of positive
roots in $\Phi_0$ and $\Phi$ corresponding to $\b_0$ and $\b$.
Let $\rho := \frac{1}{2}\sum_{\alpha \in \Phi^+} \alpha$. 
For $\la \in \t^*$, let $\C_{\la - \rho}$ be the
one dimensional $\t$-module of weight $\la - \rho$.
Let
$$
M(\la) := U(\g) \otimes_{U(\b)} \C_{\la-\rho}
$$ 
denote the usual Verma
module for $\g$ of highest weight $(\la - \rho)$,
with unique irreducible quotient
$L(\la)$.
Note by Corollary~\ref{hascc} and Lemma~\ref{ks} that the central character
$\psi^\La$ of
$L(\La,e)$ is equal to the central character of $L(\la)$
for any $\la \in \La$.
All other notation used below is as explained in
$\S$\ref{sO}.

\begin{Conjecture}\label{c1}\em
For $\La \in \mathcal L$, pick $\la \in \La$ 
such that 
$(\la|\alpha^\vee) \notin \Z_{> 0}$ for
each $\alpha \in \Phi_0^+$. Then 
$L(\La,e)$ is finite dimensional if and only if
$$
\VA(\ann_{U(\g)}(L(\la))) = \overline{G \cdot e}.
$$
\end{Conjecture}

We will verify this conjecture in type $A$ (for the
standard choice of positive roots) in Corollary~\ref{cls} below.
To formulate a stronger conjecture which was inspired by ideas of Premet, let
$$
\u := \n_0 \oplus \g_+,
$$
which is a maximal nilpotent subalgebra of $\g$ contained in $\q$.
Note that
$\chi$ restricts to a character of $\u$. 
Let $\mathcal{O}(\chi) =\mathcal{O}(\chi;\t,\q)$
denote the category of all finitely generated $\g$-modules $M$
that are locally finite over $Z(\g)$ and semisimple over $\t^e$, such that
$x-\chi(x)$ acts locally nilpotently on $M$ for all $x \in \u$.
This is the category $\mathcal{N}(\chi)$ from \cite{MS}
(with $\n$ there equal to our $\u$) except we have added the 
mild extra condition that the center of the Levi factor of $\q$ containing 
$\t$ acts semisimply. In the case $\chi = 0$ we note that
$\mathcal O(\chi)$ is the usual BGG category $\mathcal O$ again.
To define the basic objects in the category $\mathcal O(\chi)$, 
let $R$ denote the quotient of $U(\g)$ by the left ideal
generated by all $\{x-\chi(x)\:|\:x \in \u\}$.
This left ideal is invariant under right multiplication by
elements of $U(\g_0,e)$, hence $R$ is a $(U(\g), U(\g_0,e))$-bimodule.
For $\La \in \mathcal L$, set
$$
M(\La,\chi) := R \otimes_{U(\g_0,e)} V_\La,
$$
naturally an object of $\mathcal O(\chi)$ of central character $\psi^\La$.
In \cite[$\S$2]{MS}, it is shown that $M(\La,\chi)$
has a unique irreducible quotient $L(\La,\chi)$, and that every object in
$\mathcal O(\chi)$ has a composition series involving only the $L(\La,\chi)$ 
as composition factors.

\begin{Conjecture}\label{c2}\em
There is an equivalence of categories
$\mathbb{W}:\mathcal{O}(\chi) \rightarrow \mathcal{O}(e)$
such that 
$\mathbb{W} M(\La,\chi) \cong M(\La,e)$ and
$\mathbb{W} L(\La,\chi) \cong L(\La,e)$ 
for
each $\La \in \mathcal L$.
Moreover, $\mathbb{W}$ should respect annihilators in the sense that
$$
\ann_{U(\g)}(M) = \ann_{U(\g)}(Q_\lag \otimes_{U(\g,e)} \mathbb{W} M)
$$
for each $M \in \mathcal{O}(\chi)$.
\end{Conjecture}
We point out that Conjecture~\ref{c1} follows from Conjecture~\ref{c2}.
Indeed, using \cite[Theorem 5.1]{MS}, one can check 
for $\la \in \La$ as in Conjecture~\ref{c1} that
$$
\ann_{U(\g)} (L(\La,\chi)) = \ann_{U(\g)}(L(\la)).
$$
By Conjecture~\ref{c2} we get that
$\ann_{U(\g)}(Q_\lag \otimes_{U(\g,e)} L(\La,e)) = \ann_{U(\g)}(L(\la))$,
and then Conjecture~\ref{c1} follows
using 
\cite[Theorem 3.1]{P2}
and \cite[Theorem 1.2.2(ii),(ix)]{Losev} (see the discussion just before
Theorem~\ref{finiteness}).
Combined with \cite[Theorem 6.2]{Back} and the 
Kazhdan--Lusztig conjecture for $\g$,
Conjecture~\ref{c2} 
would also mean that the composition multiplicities of all
Verma modules $M(\La,e)$
can be computed in terms of Kazhdan--Lusztig polynomials. In particular,
the Kazhdan--Lusztig conjecture of \cite{deVosvanDriel} (as we understand it)
is a consequence, as is \cite[Conjecture 7.17]{BK} in type A.
Note finally that Conjecture~\ref{c2} (hence also Conjecture~\ref{c1})
is true if $e \in \g$ is a long root element. In this special case
for the good grading arising from the
$\ad h$-eigenspace decomposition of $\g$,
the equivalence of categories
$\mathbb{W}$ is given simply by taking Whittaker vectors with respect
to $\m \oplus \k_+$; see
\cite[Theorem 7.1]{P2}.

\subsection{Type A}\label{sA}

We now recast some of the
results of \cite{BK} in the language
of this paper. In 
particular we prove Conjecture \ref{c1} for the
standard choice of positive roots in type A.
So let $\g := \gl_N(\C)$ equipped with the trace form $(.|.)$, 
$\t$ be the set of diagonal matrices
and $\b$ be the set of upper triangular matrices.
Let $\eps_i \in \t^*$ 
be the $i$th diagonal coordinate function.
Then the root system is $\Phi=\Phi^+ \sqcup (-\Phi^+)$ where
$\Phi^+ := \{\eps_i-\eps_j\:|\:1 \leq i < j \leq N\}$ as usual.

Let $\mathbf p$ be a partition of $N$ and draw its Young diagram 
like in the following example:
$$
\Diagram{1&2\cr3&4&5\cr6&7&8&9\cr}
\begin{picture}(0,5)
\put(-68.2,13.6){$p_1$}
\put(-68.2,.6){$p_2$}
\put(-68.2,-12.4){$p_3$}
\put(-44, -30){$q_1\, q_2\, q_3\, q_4$}
\end{picture}
$$

\vspace{8mm}

\noindent
We let $n$ denote the number of rows and $\ell$ denote the number of
columns in the Young diagram of $\mathbf p$. We index the rows of the diagram
by $1,\dots,n$ from top to bottom, columns by $1,\dots,\ell$
from left to right, and boxes by $1,\dots,N$ along rows as in the example. 
Let $p_i$ (resp.\ $q_i$) denote the number of boxes in the $i$th row 
(resp.\ $i$th column).
Let $\RR(i)$ and $\CC(i)$ denote the row and
column numbers of the $i$th box.
Letting $e_{i,j}$ denote the $ij$-matrix unit, we let 
$e \in \g_2$ be the nilpotent matrix
$$
e = \sum_{\substack{1 \leq i,j \leq N \\ \RR(i) = \RR(j)\\ \CC(i) = \CC(j)-1}} e_{i,j},
$$
which clearly has 
Jordan type $\mathbf p$; 
e.g. $e = e_{1,2} + e_{3,4}+e_{4,5}+e_{6,7}+e_{7,8}+e_{8,9}$ 
in the above example.
We define an even good grading for $e$ 
by declaring that $e_{i,j}$ is of degree $2 (\CC(j)-\CC(i))$. 
We call this the {\em standard good grading}.
Now define the finite $W$-algebra $U(\g,e)$ as in
$\S$\ref{snonlinear}. As the good grading is even, $U(\g,e)$
is simply a subalgebra of $U(\p)$. The Levi factor $\h$
of $\p$ satisfies
$$
\h \cong \gl_{q_1}(\C) \oplus\cdots\oplus\gl_{q_\ell}(\C).
$$
We also fix the choice of the parabolic $\q$
in $\S$\ref{sCartan} to be the span of the matrix units
$\{e_{i,j}\:|\:\RR(i) \leq \RR(j)\}$. So the Levi factor $\g_0$
of $\q$ satisfies
$$
\g_0 \cong \gl_{p_1}(\C)\oplus\cdots\oplus \gl_{p_\ell}(\C).
$$
The choice of $\q$ determines a system of positive
roots $\Phi^e_+$ in the restricted root system $\Phi^e$, which we call
the {\em standard positive roots}. 

We incorporate the following two shifts 
as indicated in Remarks~\ref{thetam} and \ref{vary}:
$$
\eta := \sum_{i=1}^N (n-q_{\CC(i)}-q_{\CC(i)+1}-\cdots-q_\ell) \eps_i,\qquad
o := -{\textstyle\frac{1}{2}}(N-1)\sum_{i=1}^N \eps_i.
$$
Noting $\eta$ does indeed extend to a character of $\p$,
the embedding $\theta$ from Theorem~\ref{lif} shifted in this way is
the restriction of $S_\eta:U(\p) \rightarrow U(\p)$,
which matches \cite[(9.2)]{BKAdv}.
Also the choice of origin $o$ means that the weight $\rho$
from $\S$\ref{scharacters} is
$$
\rho = -\eps_2-2\eps_3-\cdots-(N-1)\eps_N,
$$
which agrees with the choice made in \cite{BK}.
In \cite{BKAdv} an {\em explicit} linear map
$\Theta:\g^e\hookrightarrow U(\g,e)$ as in
Theorem~\ref{Th} was described. 
The images of a certain distinguished basis
of $\g^e$ under this explicit map $\Theta$ were denoted
\begin{eqnarray*}
&&\{D_i^{(r)}\in U(\g,e)\:|\: 1\leq i\leq n,\ 1\leq r\leq p_i\},\\
&&\{E_{i,j}^{(r)}\in U(\g,e)\:|\: 1\leq i<j\leq n,\ p_j-p_i<r\leq p_j\},\\
&&\{F_{i,j}^{(r)}\in U(\g,e)\:|\: 1\leq i<j\leq n,\ 0<r\leq p_i\}.
\end{eqnarray*}
These elements belong to the zero, positive and negative restricted root
spaces of $U(\g,e)$, respectively. 
Recall the maps $\xi_{-\epsilon}$ from Lemma~\ref{ks} and
$\pi_{-\gamma}$ from Theorem~\ref{csa}.

\begin{Lemma}\label{last}
$\xi_{-\epsilon}(\pi_{-\gamma}(D_i^{(r)}))$ is equal to the
$r$th elementary symmetric function in 
$\{e_{j,j}+i-1\:|\:1 \leq j \leq N, \RR(j) = i\}$.
\end{Lemma}

\begin{proof}
We need to recall the explicit form of the element $D_i^{(r)}$
from \cite[Corollary 9.4]{BKAdv}:
$$
D_i^{(r)} = \sum_{s = 1}^r
\sum_{\substack{i_1,\dots,i_s\\j_1,\dots,j_s}}
(-1)^{r-s+\#\{1 < k \leq s\:|\:\RR(i_k) < i\}}
S_\eta(e_{i_1,j_1} \cdots e_{i_s,j_s})
$$
where
the second sum is over all $1 \leq i_1,\dots,i_s,j_1,\dots,j_s \leq N$
such that
\begin{itemize}
\item[(1)] $\CC(j_1)-\CC(i_1)+\cdots+\CC(j_s)-\CC(i_s) = r-s$;
\item[(2)] $\CC(i_t) \leq \CC(j_t)$ for each $t=1,\dots,s$;
\item[(3)] if $\RR(j_t) \geq i$ then
$\CC(j_t) < \CC(i_{t+1})$ for each
$t=1,\dots,s-1$;
\item[(4)]
if $\RR(j_t) <i$ then $\CC(j_t) \geq \CC(i_{t+1})$
for each
$t=1,\dots,s-1$;
\item[(5)] $\RR(i_1) = \RR(j_s)=i$;
\item[(6)]
$\RR(j_t)=\RR(i_{t+1})$ for each $t=1,\dots,s-1$.
\end{itemize}
We claim that 
the map $\pi:U(\p)_0 \rightarrow U(\p_0)$ annihilates all 
$S_\eta(e_{i_1,j_1}\cdots e_{i_s,j_s})$
on the right hand side of this formula 
that have $\RR(i_t) \neq i$ for some $t$.
To see this, take such a monomial and the maximal such $t$.
If $\RR(i_t) < i$ then $e_{i_t,j_t}$
can be commuted to the end of the
monomial in view of (3), hence 
since it lies in a positive restricted root
space
it is mapped to zero by $\pi$.
If $\RR(i_t) > i$ then we let $1 \leq u < t$ be 
maximal such that $\RR(i_u) < \RR(j_u)$.
Again $e_{i_u,j_u}$ can be commuted to the end of the monomial by (3) and 
$\pi$ gives zero.

Using the claim and (6) we see that $\pi(D_i^{(r)})$ is 
given explicitly by the analogous expression
summing over 
$1 \leq i_1,\dots,i_s,j_1,\dots,j_s \leq N$ satisfying the same conditions as before and also
$\RR(i_t) = \RR(j_t) = i$ for all $t$.
Applying $S_{-\gamma}$
then 
$S_{-\epsilon} \circ \xi$ and using Lemma~\ref{bed}
(recalling that $\theta$ is the restriction of $S_\eta$)
we see that 
$\xi_{-\epsilon}(\pi_{-\ga}(D_i^{(r)}))$
is equal to the $r$th elementary symmetric function in
$$
\{S_{-\epsilon-\delta}(e_{j,j})\:|\:1 \leq j \leq N, \RR(j) = i\}.
$$
It remains to show $S_{-\epsilon-\delta}(e_{j,j})
= e_{j,j}+i-1$.
To see this, let
$\N(j), \NE(j), \E(j), \dots$ denote
 the number of boxes to the north (strictly above and in the same column), 
north east (strictly above
and strictly to the right), east (strictly to the right and in the same row),
\dots\, of the $j$th box.
The weights $\delta$ from $\S$\ref{sCartan}
and $\epsilon$ from $\S$\ref{sLevi} are then given explicitly by the
formulae
\begin{align*}
\delta &= \sum_{j=1}^N (\NW(j)+\N(j)+\NE(j)+\E(j)+\SS(j)+1-n)\eps_j,\\
\epsilon &= -\sum_{j=1}^N (\NW(j)+\N(j)+\NE(j)+\E(j))\eps_j,
\end{align*}
recalling we have shifted by $-\eta$ and $o$ as indicated in
Remarks~\ref{thetam} and \ref{vary}.
Hence
$$
\epsilon+\delta =\sum_{j=1}^N (\SS(j)+1-n) \eps_j = \sum_{j=1}^N(1-\RR(j))\eps_j
$$
as required to complete the proof.
\end{proof}

A {\em $\mathbf p$-tableau} means 
a filling of boxes of the Young diagram of $\mathbf p$
with complex numbers. 
The map sending
a tableau to
the weight $\sum_{i=1}^N a_i \eps_i$,
where $a_i$ is the entry in the $i$th box,
defines a bijection from
the set $\Tab(\mathbf p)$ of all $\mathbf p$-tableaux
to the set $\t^*$.
It induces a bijection from
the set $\RRow(\mathbf p)$ of all row equivalence
classes of $\mathbf p$-tableaux
to the set $\mathcal L = \t^* / W_0$
from $\S$\ref{sLevi}.
Let $\leq$ denote the partial order on $\C$ 
defined by $a \leq b$ if $b-a \in \Z_{\geq 0}$.
We call a tableau {\em column strict} if its
entries are strictly increasing up columns from bottom to top 
in this order. 

\begin{Theorem}\label{fdclass}
Let $\La \in \mathcal L$ and $A \in \RRow(\mathbf p)$ be the
corresponding row equivalence class of $\mathbf p$-tableaux.
Then
$L(\La, e)$ is finite dimensional if and only if $A$
has a column strict representative.
\end{Theorem}

\begin{proof}
In \cite[$\S$6.1]{BK} a $U(\g,e)$-module $M$ is called a
highest weight module of type $A$ if it is generated
by a vector $v_+$ that is annihilated by all the
$E_{i,j}^{(r)}$ and such that
$D_i^{(r)}$ acts on $v_+$ by multiplication by
the $r$th elementary symmetric function in 
the elements $\{a_j + i-1\:|\:1 \leq j \leq N, \RR(j) = i\}$,
where $a_j$ is the entry in the $j$th box of some representative of $A$.
In view of Lemma~\ref{last} and the explicit definition
of $V_\La$ given just after Lemma~\ref{ks}, 
this is exactly the same as the notion of 
a highest weight module of type $\La \in \mathcal L$ from $\S$\ref{sVerma}.
Hence the Verma modules $M(\La, e)$ and their irreducible quotients
$L(\La,e)$ here are exactly the same as the modules $M(A)$
and $L(A)$ in \cite[$\S$6.1]{BK}.
Given this, the present theorem is a restatement of \cite[Theorem 7.9]{BK}.
\end{proof}

\begin{Corollary}\label{cls}
Conjecture~\ref{c1} holds in the present situation.
\end{Corollary}

\begin{proof}
To deduce this from Theorem~\ref{fdclass}, we need to recall
some classical results describing the associated varieties of 
primitive ideals in $U(\g)$ in terms the
Robinson-Schensted correspondence.
Let $\la = \sum_{i=1}^N a_i \eps_i \in \t^*$.
We define a tableau $A(\la)$ by
starting from the empty tableau 
and then using the Robinson-Schensted row insertion algorithm
to successively incorporate
the complex numbers $a_1,\dots,a_N$.
At the $i$th step we add $a_i$ to the bottom row of the 
tableau unless there is an entry $b$ already in the bottom row with
$a_i < b$, in which case we pick the smallest such $b$,
replace $b$ by $a_i$ then bump $b$ into the next row up by
the analogous procedure.
See \cite[$\S$1.1]{fulton} for a detailed account.
By \cite[Corollary 3.3]{Jo2} (together with \cite[Lemma 2.4]{Jo2} 
to reduce to regular weights) it is known that
$\VA(\ann_{U(\g)}(L(\la)))$ is equal to the closure of the
orbit consisting of all nilpotent matrices of Jordan type equal
to the shape of the tableau $A(\la)$.

Now to prove the corollary we take $\La \in \mathcal L$
and pick a representative $\la \in \La$ such that
$(\la|\alpha^\vee) \notin \Z_{> 0}$ for all $\alpha \in \Phi_0^+$.
Let $A$ be the corresponding $\mathbf p$-tableau.
Thus if $a < b$ are entries in the same row of $A$ then $a$ is 
located to the left of $b$.
It is now an elementary combinatorial
exercise to check that 
$A$ is row equivalent to a column strict tableau
if and only if $A(\la)$ is of shape $\mathbf p$.
Combined with Theorem~\ref{fdclass} and the preceding paragraph, 
we deduce that $L(\La,e)$ is finite dimensional
if and only if 
$\VA(\ann_{U(\g)}(L(\la))) = \overline{G \cdot e}$.
\end{proof}

The result just proved 
also holds for an arbitrary good grading; the general case
easily reduces to the standard good grading considered here
using \cite[Theorem 2]{BG}.

\end{document}